\documentclass[a4paper,reqno]{amsart}

\usepackage[british]{babel}

\numberwithin{equation}{section}

\usepackage{amssymb}
\usepackage{amsmath,amsthm} % Math packages
\usepackage[pdftex]{graphicx}	
\usepackage{url}
\usepackage[title,titletoc,page]{appendix} 
\usepackage{listings} 
\usepackage{color}
\usepackage{comment}
\usepackage{caption}
\usepackage{subcaption}
\usepackage{algorithm}

\usepackage{fancyhdr}
\numberwithin{equation}{section}		% Equationnumbering: section.eq#
\numberwithin{figure}{section}			% Figurenumbering: section.fig#
\numberwithin{table}{section}				% Tablenumbering: section.tab#

\newcommand{\R}{\mathbb{R}}
\newcommand{\N}{\mathbb{N}}

\usepackage[left=3.5cm,right=3.5cm,top=3.5cm,bottom=3.5cm,footskip=1.5cm,headsep=1cm]{geometry}
\usepackage[%
pdfauthor={Bryn Davies and Clemens Thalhammer},%
pdftitle={Convergence of supercell and superspace methods for computing spectra of quasiperiodic operators},
pdfsubject={Convergence of supercell and superspace methods for computing spectra of quasiperiodic operators},
pdfkeywords={quasicrystal, cut and project, fractal spectrum, Cantor set, Fibonacci tiling, almost Mathieu operator},
]{hyperref}

\usepackage[T1]{fontenc} % To switch to the T1 encoding
\usepackage{lmodern} % To switch to Latin Modern
\rmfamily % To load Latin Modern Roman and enable the following NFSS declarations.
\DeclareFontShape{T1}{lmr}{b}{sc}{<->ssub*cmr/bx/sc}{}
\DeclareFontShape{T1}{lmr}{bx}{sc}{<->ssub*cmr/bx/sc}{}

\usepackage{lipsum}
\usepackage{mathtools}
\usepackage{amsmath}
\usepackage{amssymb}
\usepackage{tikz}
\usetikzlibrary{matrix,positioning,decorations.pathreplacing,calc}
\usetikzlibrary{
decorations.text,%
decorations.markings,%
shadows}
\usepackage{adjustbox}
\usepackage{graphicx}

\usepackage[format=hang]{caption}
\usepackage{subcaption}
\usepackage[normalem]{ulem}

\usepackage{bm}

\usepackage[
sorting=nty,
maxnames=99,
maxalphanames=5,
natbib=true,
sortcites]{biblatex}

\usepackage{csquotes}

\DeclareNameAlias{default}{family-given}

\graphicspath{{figures/}}

\AtEveryBibitem{% Clean up the bibtex rather than editing it
 \clearfield{url}
 \clearfield{issn}
 \clearfield{isbn}
 \clearfield{urldate}
 
 \ifentrytype{book}{
  \clearfield{pages}}{% 
 }
 }

\usepackage{xargs}                      % Use more than one optional parameter in a new commands
\usepackage[prependcaption,textsize=tiny,textwidth=3cm]{todonotes}
\newcommandx{\unsure}[2][1=]{\todo[linecolor=red,backgroundcolor=red!25,bordercolor=red,#1]{#2}}
\newcommandx{\change}[2][1=]{\todo[linecolor=blue,backgroundcolor=blue!25,bordercolor=blue,#1]{#2}}
\newcommandx{\info}[2][1=]{\todo[linecolor=OliveGreen,backgroundcolor=OliveGreen!25,bordercolor=OliveGreen,#1]{#2}}
\newcommandx{\improvement}[2][1=]{\todo[linecolor=black,backgroundcolor=black!25,bordercolor=black,#1]{#2}}
\newcommandx{\thiswillnotshow}[2][1=]{\todo[disable,#1]{#2}}
\usepackage{amsthm}
\usepackage[noabbrev,capitalize]{cleveref}
\crefname{proposition}{Proposition}{Propositions}
\crefname{equation}{}{}

\newtheorem{theorem}{Theorem}[section]
\newtheorem{lemma}[theorem]{Lemma}
\newtheorem{proposition}[theorem]{Proposition}
\newtheorem{corollary}[theorem]{Corollary}

\theoremstyle{definition}
\newtheorem{definition}[theorem]{Definition}

\crefname{assumption}{Assumption}{Assumptions}
\crefname{definition}{Definition}{Definitions}
\crefname{corollary}{Corollary}{Corollaries}
\crefname{enumi}{item}{items}
\usepackage{fancyhdr}

\fancypagestyle{plain}{%
\fancyhead[C]{}
\fancyfoot[C]{}

}
\fancypagestyle{nosection}{%
\fancyhead[CE]{}
\fancyhead[CO]{}
\fancyfoot[CE,CO]{\thepage}
}
\DeclareMathOperator{\Z}{\mathbb{Z}}
\DeclareMathOperator{\Q}{\mathbb{Q}}

\renewcommand{\i}{\mathrm{i}}

\renewcommand{\qedsymbol}{$\blacksquare$}

\renewcommand{\d}{\ \mathrm{d}}

\usepackage{fancyhdr}

\fancypagestyle{plain}{%
\fancyhead[C]{}
\fancyfoot[C]{}

}
\fancypagestyle{nosection}{%
\fancyhead[CE]{}
\fancyhead[CO]{}
\fancyfoot[CE,CO]{\thepage}
}
\renewcommand{\epsilon}{\varepsilon}

\renewcommand{\i}{\mathrm{i}}

\usepackage{tcolorbox}
\usetikzlibrary{tikzmark,calc}

\usepackage{enumerate}
\usepackage{upgreek}

\begin{document}

\title[Title
]{Title}

\title[Convergence of supercell and superspace methods]{Convergence of supercell and superspace methods for computing spectra of quasiperiodic operators}

 \author[B. Davies]{Bryn Davies}
\address{\parbox{\linewidth}{Bryn Davies\\
Mathematics Institute, University of Warwick, Coventry CV4 7AL, UK}}
\email{bryn.davies@warwick.ac.uk}
\thanks{}

\author[C. Thalhammer]{Clemens Thalhammer}
 \address{\parbox{\linewidth}{Clemens Thalhammer\\
 ETH Z\"urich, Department of Mathematics, Rämistrasse 101, 8092 Z\"urich, Switzerland.}}
\email{cthalhammer@student.ethz.ch}

\begin{abstract}
We study the convergence of two of the most widely used and intuitive approaches for computing the spectra of differential operators with quasiperiodic coefficients: the supercell method and the superspace method. In both cases, Floquet-Bloch theory for periodic operators can be used to compute approximations to the spectrum. We illustrate our results with examples of Schr\"odinger and Helmholtz operators.
\end{abstract}

\maketitle

\vspace{3mm}
\noindent
\textbf{Keywords.} quasicrystal, cut and project, fractal spectrum, Cantor set, Fibonacci tiling, almost Mathieu operator\\

\section{Introduction}

Quasicrystals have long been discussed as an exciting new direction for photonic crystals and metamaterials, as their self-similar but aperiodic geometries lead to exotic wave-structure interactions beyond those possible with either periodic or random materials \cite{zolla1998remarkable, Simon1982}. The source of this excitement is mostly the exotic spectral properties that can be exhibited by differential operators with quasiperiodic coefficients, such as fractal and Cantor spectra \cite{Simon1982, avila2009ten}, metal-insulator transitions \cite{jitomirskaya1999metal} and unexpected symmetries in reciprocal space \cite{shechtman1984metallic}. There is also interest in their effective properties \cite{bouchitte2010homogenization, wellander2018two}, such as the ability to achieve isotropy in lattice structures \cite{chen2020isotropic} and negative refraction \cite{morini2019negative}. These properties have, so far, been put to use in applications such as building wave guides \cite{davies2022symmetry, beli2023interface} and rainbow devices \cite{davies2023graded}.

The difficulty with working with differential operators with quasiperiodic coefficients  is that their spectra are notoriously hard to compute. This is the main barrier to the wider uptake of quasicrystalline metamaterials for applications. While some efficient and error-controlled algorithms have been developed , \emph{e.g.} \cite{Colbrook2019, johnstone2022bulk}, the relative challenge of understanding and implementing these new algorithms has meant that their usage in the applied physical community has been limited. Instead, researchers have preferred simpler, intuitive approaches, even if their convergence properties are less well known. 

The most widely used approach in the physical literature for approximating the spectra of non-periodic systems is the \emph{supercell method}. This method constructs periodic versions of the operator's coefficients by taking a large, finite-sized section of the domain and repeating those values of the coefficients periodically. This is convenient as the periodic version of the operator has a well-behaved spectrum, composed of a countable union of continuous spectral bands, which can be computed easily using Floquet-Bloch theory \cite{kuchment2016overview}. As a result, the supercell method underpins many of the breakthroughs in quasiperiodic metamaterials, see \emph{e.g.} \cite{chan1998photonic, florescu2009complete, davies2022symmetry, morini2019negative, davies2023graded, hamilton2021effective}.

An intuitive conjecture about the supercell method is that better approximations can be obtained by increasing the size of the unit cell. However, in what sense the resulting sequence of spectra converges to that of the limiting quasiperiodic operator is far from obvious. The sequence of Floquet-Bloch spectra are each composed of a countable union of spectral bands. Conversely, the limiting spectra can, in general, have many different exotic properties, such as the gaps being dense and the spectrum being a Cantor set \cite{Simon1982}. Thus, it's apparent that taking this limit involves nontrivial behaviour. Fortunately, there is an established theory for approximating almost periodic functions by periodic functions \cite{Shubin1978}, which can be used to establish spectral convergence results \cite{Avron1990, Damanik2016}.

Although the behaviour of the sequence of approximate spectra obtained from supercell approximations with increasingly large unit cells is difficult to understand in general, it has been observed that many of the main spectral gaps emerge relatively early in the sequence and persist. These were coined \emph{super band gaps} by \cite{Morini2018} and have been studied in several settings. For example, they were recently proved to exist in one-dimensional systems based on Fibonacci-tilings \cite{davies2023super} and have been observed in two-dimensional lattices based on Penrose tilings \cite{damanik2023discontinuities}. The convergence theory discussed in this work shows that when a super band gap exists in the sequence of supercell approximations, this is guaranteed to correspond to a gap in the spectrum of the limiting quasiperiodic operator.

The other approach explored in this work is the \emph{superspace method}. This exploits the fact that, although quasicrystals are generally non-periodic, they have underlying periodicity in the sense that they can be obtained by an incommensurate projection from a higher-dimensional periodic space \cite{janot2012book}. Using this approach to compute spectra was proposed by \cite{Rodriguez2008}, although they observed spectral pollution by spurious eigenvalues in many of the spectral gaps and it remains to develop a convergence theory for the method. The superspace principle has also been used successfully by others, to perform direct numerical simulations of forced problems \cite{amenoagbadji2023wave} and to derive effective medium theories through asymptotic homogenisation \cite{bouchitte2010homogenization, wellander2018two}.

Particularly in applied physical and engineering settings, many researchers have overcome the challenges of computing spectra of quasicrystals by resorting to only ever simulating finite-sized structures. In this case, one can either compute the countable collection of eigenvalues \cite{xia2020topological, kraus2012topological, MartSabat2021, ni2019observation} or estimate a local density of states (LDOS) \cite{della2005band, asatryan2001two}. These approaches typically give reliable predictions of the experimental behaviour of devices. In certain cases (especially those which are periodic), the spectra of finite-sized systems are known to converge to their infinite counterparts \cite{ammari2023spectral}, so it's a reasonable expectation that there should be some correspondence between the spectra of quasiperiodic operators on finite and infinite domains. However, spectra are often polluted by edge modes, which need to be properly understood, and it is not always clear how properties will scale or generalise. Further, these methods will never capture the unique mix of long-range aperiodicity and self-similarity present in quasicrystals. The motivation for developing rigorous convergence theories for supercell and superspace methods is that these approaches would facilitate more systematic investigation of quasiperiodic metamaterials, analogous to how any study of periodic systems typically begins by examining the spectral band diagram to understand the system's properties.

We will begin by recalling some mathematical preliminaries in Section~\ref{sec:prelims}, such as the required background on almost periodic functions and tiling rules, before examining the supercell and superspace methods in Sections~\ref{sec: supercell} and ~\ref{sec: superspace}, respectively. Finally, in Section~\ref{sec:interface} we will show how our theory can be applied to study localised interface modes in quasicrystals with defects; we prove the existence of such modes and give estimates for their eigenfrequencies and (exponential) decay rates.

\section{Mathematical preliminaries} \label{sec:prelims}

\subsection{Problem setting}
Suppose we want to compute the spectrum of a self-adjoint elliptic operator $A$ with quasiperiodic coefficients of the form
\begin{equation}\label{eq: operator_definition}
    A= \sum_{k}a_k(x) \frac{d^k}{dx^k},
\end{equation}
where the coefficient $a_k(x)$ are assumed to be quasiperiodic and smooth. By \cite{Shubin1978}, this amounts to finding all pairs $(\lambda, u) \in \R \times \mathcal{C}^\infty_b(\R)$ such that
\begin{equation}\label{eq: eigenvalue_problem}
    Au = \lambda u.
\end{equation}
If the coefficient $a_k$ are periodic, then Floquet-Bloch theory applies and \eqref{eq: eigenvalue_problem} can be reduced to a collection of eigenvalue problems for $(\lambda_\alpha, u_\alpha) \in \R \times \mathcal{C}^\infty((0,T))$ given by
\begin{equation}\label{eq: eigenvalue_problem_periodic}
    \begin{cases}
        Au_\alpha = \lambda_\alpha u_\alpha\\
        u_\alpha(T) = e^{\i\alpha T}u_\alpha(0),
    \end{cases}
\end{equation}
where $T$ is the length of the periodic and $\alpha \in [0, \frac{2\pi}{T})$ is the quasi-momentum. From a computational point of view, \eqref{eq: eigenvalue_problem_periodic} is much simpler than \eqref{eq: eigenvalue_problem}, as elliptic operators on compact domains have a countable set of eigenvalues. While integral decomposition techniques do exist for quasiperiodic (or almost periodic) differential operators \cite{Bellissard:1981ju}, the resulting domains are still infinite so Floquet-Bloch is not directly applicable. Related to this, the spectrum can take on exotic forms. For example, there are instances of elliptic differential operators with almost periodic coefficients whose spectrum is Cantor-like \cite{Moser1981}.

We will consider eigenvalue problems \eqref{eq: eigenvalue_problem} for operators of the form \eqref{eq: operator_definition} with coefficients $a_k(x)$ that are smooth and are either quasiperiodic (in the sense of Definition~\ref{defn:quasiperiodic}) or formed by a tiling rule (as introduced in Section~\ref{subsec: tilings}). In Section~\ref{sec: supercell}, we will consider the supercell method. We will prove how the Floquet-Bloch spectra converge to the spectrum of the limiting operator in the limit of large unit cell length. We will present versions of these results for the specific case of tiling-based operators and present numerical examples for Schr\"odinger and Helmholtz-type operators. In Section~\ref{sec: superspace}, we will consider the superspace method, as proposed in \cite{Rodriguez2008}. We will show that the lifted operator has the same spectrum and present some numerical examples. These examples predict spectra that agree with the supercell predictions, up to some spectral pollution by some spurious eigenvalues. This spectral pollution was also observed in \cite{Rodriguez2008}; we show how it depends on the choice of discretisation (and is a particularly troublesome issue in the case of the plane wave expansion method).

\subsection{Almost periodic functions}
In this section we will provide a brief introduction to almost periodic functions and operators. For a more thorough overview, see \cite{Shubin1978,Simon1982,Corduneanu2009}. The most straightforward way to introduce almost periodic functions is through trigonometric polynomials. We define $\text{Trig}(\R^n)$ as the space of trigonometric polynomials over $\R^n$, i.e. functions that are finite linear combinations of $\exp (\i\langle\lambda, x\rangle)$, where $\lambda \in \R^n$.
\begin{definition}
    We define the set of uniformly almost periodic function $CAP(\R)$ as the closure of $\text{Trig}(\R^n)$ with respect to the supremum norm. We further denote $CAP^k(\R^n) = CAP(\R^n) \cap \mathcal{C}^k(\R^n)$, equipped with the $\mathcal{C}^k$ norm.
\end{definition}
The following Theorem from \cite[Theorem~2.3]{Shubin1978} highlights an important property of almost periodic functions: that they can be approximated arbitrarily well by periodic functions on compact domains.
\begin{theorem}\label{thrm: periodic_approximation}
    Let $\mathcal{F}$ be a precompact subset of $CAP^k(\R^n)$ and, for $R>0$, let $Q_R\subset \R^n$ be the set
    \begin{equation*}
        Q_R = \left\{x \in \R^n\mid 0 \leq x_i  \leq R\right\}.
    \end{equation*}
    Then, for any $\epsilon>0$, $r>0$ and $T_0>0$, there exists $T\geq T_0$ and a continuous linear operator $L$ with finite dimensional image such that $Lf$ is a $T$-periodic trigonometric polynomial and 
    $$\sup_{x \in Q_{rT}}\lvert f(x)- Lf(x)\rvert<\epsilon,$$
    for any $f \in \mathcal{F}$.
\end{theorem}
Many fundamental results about almost periodic functions rely on the fact they can be approximated by periodic functions, as stated in Theorem \ref{thrm: periodic_approximation}, and it is the foundation of the supercell method considered here. 

To define the equivalent of $L^p$ spaces for almost periodic functions, we will first need to introduce the mean value functional.
\begin{definition}\label{prop: mean_value}
    For $f \in CAP(\R^n)$, we define the mean value of $f$ by
    \begin{equation*}
        M[f] = \lim_{R \rightarrow\infty} \frac{1}{R^n} \int_{Q_R + s} f(x) \d x,
    \end{equation*}
    where $s \in \R^n$ can be chosen arbitrarily.
\end{definition}
\begin{definition}\label{def: besicovitch_space}
    The Besicovitch space $B^2(\R^n)$ is the closure of the characters $CAP(\R^n)$ with respect to the inner product given by
    \begin{equation*}
        \langle f,g\rangle_B := M[f\Bar{g}].
    \end{equation*}
\end{definition}
From Theorem~\ref{thrm: periodic_approximation} and Definition~\ref{def: besicovitch_space}, it follows that the set $\{\exp(\i\langle \gamma, x\rangle )\}_{\gamma \in \R^n}$ forms an orthonormal basis of $B^2(\R^n)$. Given $f\in B^2(\R^n)$, we can associate to $f$ a generalised Fourier series
\begin{equation*}
    f(x) = \sum_\gamma f_\gamma e^{\i\langle \gamma, x \rangle},
\end{equation*}
where $ f_\gamma = M[fe^{-\i\langle \gamma, x \rangle}]$. We now make a definition of what it means for a function in $B^2(\R^n)$ to be quasiperiodic.

\begin{definition} \label{defn:quasiperiodic}
    We say that a function $f\in B^2(\R^n)$ is quasiperiodic if there exists some 
    $F \in L^2(\mathbb{T}^m)$, with $m>n$, and some $\zeta_1, \dots, \zeta_m \in \R^n$ which are linearly independent over $\Q$ such that 
    \begin{equation}\label{eq: quasiperiodic}
        f(x):= F(\langle x,\zeta_1\rangle,\langle x,\zeta_2\rangle, \dots,\langle x,\zeta_m\rangle).
    \end{equation}
\end{definition}

\begin{proposition}\label{prop: qp_mean}
    Let $F \in C_b(\mathbb{T}^m),m>n$ and let $\zeta_1, \dots, \zeta_m \in \R^n$ be linearly independent. Let $f$ be defined as in $\eqref{eq: quasiperiodic}$. It holds that
    \begin{equation*}
        M[f] = \int_{\mathbb{T}^m} F(y) \d y.
    \end{equation*}
\end{proposition}
As an example, consider the case $m = 2, n = 1$. We take $\zeta_1 = 1$ and $\zeta_2 = \frac{1 + \sqrt{5}}{2}$. Then, since $\zeta_1$ and $\zeta_2$ are linearly independent over $\Q$, the slice along the vector $(\zeta_1,\zeta_2)^\top$ densely fills $\mathbb{T}^m$ \cite{Rodriguez2008}. This is depicted in Figure \ref{fig: unit_cell_filled}. This, coupled with Proposition~\ref{prop: qp_mean}, serves as the inspiration for the superspace method.

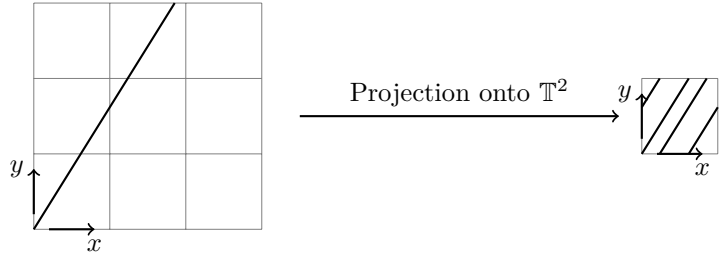
\begin{figure}
    \centering
    \begin{tikzpicture}

% Left Grid (Original Matrix X)
\draw[step=1cm,gray,very thin] (0,0) grid (3,3);
\draw[thick] (0,0) -- (3/1.618,3); % Diagonal line
\draw[->, thick] (0.2,0) -- (0.8,0) node[anchor=north] {$x$};
\draw[->, thick] (0,0.2) -- (0,0.8) node[anchor=east] {$y$};

% Right Grid (Reduced Matrix X-bar)
\begin{scope}[shift={(8,1)}]
    \draw[step=1cm,gray,very thin] (0,0) grid (1,1);
    \draw[thick] (0,0) -- (1/1.618,1); % Diagonal line
    \draw[thick] (0.61804697157,0) -- (1,0.618); % Another diagonal line
    \draw[thick] (0,0.618) -- (0.23609394314,1); % Another diagonal line
    \draw[thick] (0.23609394314,0) -- (0.85414091471,1); % Another diagonal line
    \draw[->, thick] (0.2,0) -- (0.8,0) node[anchor=north] {$x$};
\draw[->, thick] (0,0.2) -- (0,0.8) node[anchor=east] {$y$};
    
\end{scope}

\draw[->, thick] (3.5,1.5) -- (7.7,1.5) node[midway,above] {Projection onto $\mathbb{T}^2$};

    \end{tikzpicture}
    \caption{Schematic of the slice along $(\zeta_1,\zeta_2)^\top$ filling the unit cell densely}
    \label{fig: unit_cell_filled}
\end{figure}

An important result of \cite{Shubin1978} that will be important in Section \ref{sec: superspace} is the following:
\begin{theorem}\label{thrm: spectrum_equivalence}
    Let $A$ be a self-adjoint elliptic operator of the form \eqref{eq: operator_definition}. Denote $\sigma_B(A)$ the spectrum of $A$ as an operator on $CAP(\R^n)$ and $\sigma(A)$ the spectrum as an operator on $C^\infty_c(\R^n)$. It holds that
    \begin{equation*}
        \sigma_B(A) = \sigma(A).
    \end{equation*}
\end{theorem}
\begin{definition}
    Let $A$ be an elliptic self-adjoint operator with almost periodic coefficients. We define
    \begin{equation*}
        H(A) = \left\{\Tilde{A}\mid \Tilde{A}(x) = \lim_k A(x + h_k)\right\},
    \end{equation*}
    where $h_k$ is a sequence such that the limit exists uniformly in x.
\end{definition}
This set is sometimes also refered to as the \textit{frequency module} of $A$. The following alternative holds \cite{Shubin1978}:
\begin{theorem}
    $A$ is invertible or there exists some $\Tilde{A} \in H(A)$ that is not injective with domain $\mathcal{C}^\infty_b(\R^n)$.
\end{theorem}
In particular, this implies $\sigma(A) = \sigma(\Tilde{A})$ for all $\Tilde{A} \in H(A)$.

\subsection{Continued fractions} \label{sec:contfrac}

Deciding how to approximate an aperiodic function by a sequence of periodic versions will often be equivalent to approximating an irrational number by a sequence of rational numbers. A natural way to achieve this is by truncating the continued fraction representation of a number. To this end, we briefly recall the theory of continued fractions, and refer the reader to \cite{Khinchin1997} for more details. 

A \emph{simple continued fraction} is an expression of the form
\begin{equation} \label{eq:contfrac}
    a_0+ \frac{1}{a_1+\frac{1}{a_2+\dots}},
\end{equation}
where $a_0,a_1,a_2,\dots$ are positive integers, referred to as the \emph{elements} of the continued fraction. The sequence of elements can be either infinite or finite. When it is finite, we call \eqref{eq:contfrac} a \emph{finite} continued fraction. It turns out that every real number can be written uniquely in this form, as shown in \cite[Theorem~14]{Khinchin1997}:

\begin{theorem}
    To every real number $a$ there corresponds a unique continued fraction
with value equal to $a$. Furthermore, this continued fraction is finite if and only if $a$ is rational.
\end{theorem}

Given an irrational number $a$, its continued fraction representation will have an infinite sequence of elements $[a_0,a_1,a_2,\dots]$. We can construct a sequence of rational approximants indexed by $k\in\mathbb{N}$ by truncating the sequence of elements and considering the rational numbers with finite continued fractions with elements $[a_0,a_1,\dots,a_k]$. Since any such number must be rational, for each $k$ there exists coprime integers $p_k$ and $q_k$ such that the $k$\textsuperscript{th} rational number in this sequence is equal to $p_k/q_k$. Clearly, $p_k/q_k\to a$ as $k\to\infty$ and the rate of convergence can be estimated by the following result, from \cite[Theorem~10]{Khinchin1997}:

\begin{theorem}
    Let $a\in\mathbb{R}$ be equal to the infinite continued fraction with elements $[a_0,a_1,a_2,\dots]$ and let $p_k$ and $q_k$ be coprime integers such that $p_k/q_k$ is equal to the finite continued fraction with elements $[a_0,a_1,\dots,a_k]$. Then, it holds that
    $$
    \left| a-\frac{p_k}{q_k}\right| < \frac{1}{q_k q_{k+1}}.
    $$
\end{theorem}

\subsection{Tilings}\label{subsec: tilings}
Another way of generating aperiodic structures with long-range order is through the means of inflation tilings, for a thorough introduction see \cite{baake2013aperiodic}. Given a finite alphabet $\mathcal{A}_n = \left\{a_i \mid 1 \leq i \leq n\right\}$, consider the free group $F_n =  \langle a_1, \dots, a_n \rangle$ formed by the letters of $\mathcal{A}_n$ and their formal inverses together with concatenation as group operation. $F_n$ is an infinite group, containing all possible finite words that can be generated from $\mathcal{A}_n$.

\begin{definition}\label{def: substitution rule}
    A general substitution rule $\varrho$ on a finite alphabet $\mathcal{A}_n$ with $n$ letters is an endomorphism of the corresponding free group $F_n$.
\end{definition}

Substitution rules, where $\varrho(a_i)$ contains no negative powers, are of particular interest. Such substitution rules are called \textit{non-negative} substitution rules. They are often studied together with the dictionary $\mathcal{A}^*$ of $\mathcal{A}_n$, defined by
\begin{equation*}
    \mathcal{A}^* := \{ u \in F_n \mid u \text{ contains no negative powers} \}.
\end{equation*}

\begin{definition}\label{def: primitive_rule}
    A (non-negative) substitution rule $\varrho$ on a finite alphabet $\mathcal{A}_n$ is called irreducible when for each index pair $(i,j)$ there exists some $k \in \N$ such that $a_j$ is a subword of $\varrho(a_i)^k$. A (non-negative) substitution rule is called primitive if there exists some $k \in \N$ such that every $a_j$ is a subword of $\varrho(a_i)^k$.
\end{definition}
Through Abelianisation, one can associate to every substitution rule $\varrho$ a substitution matrix $M_\varrho$ such that the following diagram commutes
\begin{center}
    \begin{tikzpicture}
    \usetikzlibrary {positioning}
        \node (Free1) at (0,0)  {$F_n$};
        \node  (Free2) at (2.5,0)  {$F_n$};
        \node (Z2) at (2.5,-1.5) {$\Z^n$};
        \node (Z1) at (0,-1.5) {$\Z^n$};
        \draw[->] (Free1.east)--(Free2.west) node[midway,above] () {$\varrho$};
        \draw[->] (Free2.south)--(Z2.north)node[midway,right] () {$\iota$};
        \draw[->] (Z1.east)--(Z2.west) node[midway,above] () {$M_\varrho$};
        \draw[->] (Free1.south)--(Z1.north) node[midway,right] () {$\iota$};
    \end{tikzpicture}
\end{center}

The following lemma holds 
\begin{lemma}
    A (non-negative) substitution rule $\varrho$ is irreducible or primitive if and only if its substitution matrix $M_\varrho$ is an irreducible or primitive non-negative integer matrix, respectively.
\end{lemma}
Studying the resulting substitution matrix $M_\varrho$ can yield further insight into behaviour of the system defined by $\varrho$. For two letter substitution rules, a classification based on the Perron-Frobenius eigenvalue was developed in \cite{kolar1993new}. Notably, it can be shown that a tiling is classically quasi-crystalline (i.e. quasiperiodic) if the Perron-Frobenius eigenvalue of the substitution matrix $M_\varrho$ is a Pisot-Vijayaraghavan number \cite{BOMBIERI1986}.

Given an bi-infinite word $w$, we define the \textit{shift space} of $w$ as
\begin{equation}\label{eq: shift_space}
    \mathbb{X}(w) = \overline{\{S^iw\mid i \in \Z\}}.
\end{equation}
this is sometimes also referred to as the \textit{hull of $w$}. Another important concept is that of local indistinguishability:
\begin{definition}\label{def: locally_indistinguishable}
    Two words $u$ and $v$ in the same alphabet are locally indistinguishable (LI) and are denoted $u\sim^{LI}v,$ if every finite subword of $u$ is also a subword of $v$ and vice versa.
\end{definition}
We can now state the following theorem:
\begin{theorem}
    Every primitive substitution rule on a finite alphabet possesses a unique hull. This hull consists of a single, closed $LI$ class.
\end{theorem}
As is the case with almost periodic ellitpic differential operators, we will soon see in Section \ref{sec: supercell} that the spectrum of an operator whose coefficients are based on tiling rules is determined by its $LI$ class.

\section{Supercell method}\label{sec: supercell}
A widespread technique to study the (spectral) properties of quasicrystals is the supercell approach whereby a non-periodic coefficient is approximated by a periodically repeated sample, a so-called \emph{supercell}. This method is conceptually simple and straightforward to implement, hence its popularity. In this section, we will derive error bounds for this approach. We will further outline how one needs to choose the approximants to ensure maximal efficiency.

\subsection{Convergence theory}

To avoid notation becoming too cumbersome, we will only consider operators of the form
\begin{equation}\label{eq: operator_qp}
    A_\theta(x) = \sum_k a_k(x,\theta x) \frac{d^k}{dx^k},
\end{equation}
where $a_k \in \mathcal{C}^\infty_b(\mathbb{T}^2)$ and $\theta \in \R\setminus \Q$. The following result serves as a warm up and is a result of strong operator convergence.

\begin{proposition}\label{prop: spectrally_inclusive}
    Let $\theta \in \R$ and $\theta_n$ be a sequence converging to $\theta$. Then, it holds that
    \begin{equation*}
        \lim_{n\rightarrow \infty} \lVert A_\theta - A_{\theta_n}f\rVert_{L^2} = 0,
    \end{equation*}
    for all $f \in H^2(\R)$. In particular, the sequence $\sigma(A_{\theta_n})$ is spectrally inclusive in the sense that for every $\lambda \in \sigma(A_\theta)$ there exists a sequence $\lambda_n \in \sigma(A_{\theta_n})$ such that $\lambda = \lim \lambda_n$.
\end{proposition}
\begin{proof}
    Let $f \in C^\infty_c(\R)$ and $\epsilon>0$. We have
    \begin{equation*}
        \begin{split}
            \lVert (A_\theta - A_{\theta_n})f\rVert ^2 &= \int_\R \left\lvert \sum_k (a_k(x,\theta x) - a_k(x,\theta_nx))\frac{d^kf}{dx^k}(x)\right\rvert^2 \d x\\
            &\leq \sum_k C \int_\R \left\lvert a_k(x,\theta x) - a_k(x,\theta_nx)\right\rvert ^2 \left\lvert \frac{d^kf}{dx^k}(x) \right\rvert^2 \d x.
        \end{split}
    \end{equation*}
    As $f \in H^k$, there exists some $R>0$ such that
    \begin{equation*}
        \sum_k C \int_{\lvert x \rvert > R} \lvert a_k(x,\theta x) - a_k(x,\theta_nx)\rvert ^2 \left\lvert \frac{d^kf}{dx^k}(x) \right\rvert^2 \d x<\frac{\epsilon}{2}.
    \end{equation*}
    On the other hand, since $a_k \in \mathcal{C}_b^\infty(\mathbb{T}^2)$, for $n$ large enough, we have
    \begin{equation*}
        \sum_k C \int_{\lvert x \rvert < R} \lvert a_k(x,\theta x) - a_k(x,\theta_nx)\rvert ^2 \left\lvert \frac{d^kf}{dx^k}(x) \right\rvert^2 \d x<\frac{\epsilon}{2},
    \end{equation*}
    and the claim follows.
\end{proof}
To obtain a stronger result, we will need to construct approximate eigenfunctions. The next Lemma is a well known consequence of the spectral theorem.
\begin{lemma}\label{lem: approximate_eigenvectors}
    Let $\mathcal{H}$ be a Hilbert space and $A:D(A) \subset\mathcal{H} \rightarrow \mathcal{H}$ a normal operator. Then $\lambda \in \sigma(A)$ if and only if there exists a normalised sequence $u_n \in \mathcal{H}$ such that $\lim \lVert (A-\lambda)u_n\rVert = 0$.
\end{lemma}

The following is a generalisation of a well-known result from the study of dynamically defined Schr\"odinger operators, see e.g. \cite{Avron1990, Damanik2016}.

\begin{theorem}\label{thrm: implicit_domain_qp}
     Let $\theta \in \R$ and $A_\theta$ be a self-adjoint elliptic operator of the form \eqref{eq: operator_qp}. For $y \in \mathbb{T}^2$ we define
     \begin{equation*}
         A_{\theta,y} = \sum_k a_k(x + y_1, \theta x + y_2)\frac{d^k}{dx^k}.
     \end{equation*} Consider the set
     \begin{equation*}
         \Sigma(A_\theta) := \bigcup_{y \in \mathbb{T}^2}\sigma(A_{\theta,y}).
     \end{equation*}         
     There exists some $C>0$ such that for any $\theta' \in \R$ with $\lvert \theta - \theta'\rvert$ small enough, for every $\lambda \in \Sigma(A_\theta)$ we have
     \begin{equation*}
         d(\lambda, \Sigma(A_{\theta'})) \leq C(1 + \lvert \lambda \rvert)\lvert \theta - \theta'\rvert^{1/2}.
     \end{equation*}
\end{theorem}
\begin{proof}
  Let $\epsilon>0$ and $\lambda \in \Sigma(\theta)$. Without loss of generality, we may assume $\lambda \in A_{\theta,0} = A_\theta$. Choose $f_\lambda \in H^k(\R)$ such that $\lVert (A_\theta-\lambda)f_\lambda\rVert_{L^2}<\epsilon\lVert f_\lambda \rVert$. Let $h \in C_c^\infty([-1,1])$ such that $\lVert h \rVert_{L^2} = 1$. We further define
  \begin{equation}\label{eq: supercell_convergence_qp_1}
      h_{L,\tau}(x) = h\left(\frac{x-\tau}{L}\right).
  \end{equation}
  Using Fubini's theorem, we have that
  \begin{equation}\label{eq: supercell_convergence_qp}
      \int_\R \lVert h_{L,\tau}g \rVert^2_{L^2} \d\tau = \lVert h\rVert_{L^2}^2 \lVert g\rVert_{L^2}^2,
  \end{equation}
  for all $g \in L^2$. In particular, it holds that
  \begin{equation}\label{eq: supercell_convergence_qp_2}
    \int_\R\lVert h_{L,\tau}(A_\theta - \lambda)f_\lambda\rVert_{L^2}^2\d\tau = \epsilon^2\lVert f_\lambda \rVert^2.
  \end{equation}
  Next, we compute the commutator
  \begin{equation}\label{eq: supercell_convergence_qp_3}
      \begin{split}
          [A_\theta, h_{L,\tau}](x) &= \sum_k a_k(x) \sum_{n = 1}^{k-1}\binom{k}{n}\frac{d^nh_{L,\tau}}{dx^n}(x)\frac{d^{k-n}}{dx^{k-n}}\\
          &= \sum_k a_k \sum_{n = 1}^{k-1}\binom{k}{n}L^{-n}\frac{d^n h}{dx^n}\left(\frac{x-\tau}{L}\right)\frac{d^{k-n}}{dx^{k-n}}.
      \end{split}
  \end{equation}
  Hence, there exists some $C>0$, only depending on $h$ and $k$ such that
  \begin{equation}
      \int_\R \lVert [A_\theta, h_{L,\tau}]f_\lambda\rVert_{L^2}^2\d\tau\leq CL^{-2}\lVert f_\lambda \rVert_{H^k}^2.
  \end{equation}
  Using ellitpic regularity \cite{Taylor2011}, we have
  \begin{equation}
      \lVert f_\lambda \rVert_{H^k}^2 \leq C (\lvert\lambda\rvert + \epsilon + 1)^{2}\lVert f_\lambda \rVert^2
  \end{equation}
  For any $\delta > 0$, it now holds that
    \begin{equation}\label{eq: supercell_convergence_qp_4}
        \begin{split}
            \int_\R\lVert (A_\theta - \lambda)h_{L,\tau}f_\lambda\rVert^2\d\tau &\leq (1 + \delta)\int_R\lVert h_{L,\tau}(A_\theta - \lambda)f_\lambda\rVert_{L^2}^2\d\tau + (1 + \delta^{-1})\int_\R \lVert [A_\theta, h_{L,\tau}]f_\lambda\rVert_{L^2}^2\d\tau\\
            &\leq (1 + \delta)\epsilon^2\lVert f_\lambda\rVert^2 + (1 + \delta^{-1})CL^{-2} (\lvert\lambda\rvert + \frac{\epsilon}{2} + 1)^{2}\lVert f_\lambda \rVert^2\\
            &= \int_\R ((1 + \delta)\epsilon^2 + (1 + \delta^{-1})CL^{-2} (\lvert\lambda\rvert + \epsilon + 1)^{2})\lVert h_{L,\tau} f_\lambda \rVert^2\d\tau.
        \end{split}
    \end{equation}
  %   From \eqref{eq: supercell_convergence_qp_3} it becomes apparent that there must exist some $C>0$, which can be chosen independent of $\theta,L $ and $\tau$, such that
  %   \begin{equation*}
  %       \lVert[A_\theta, h_{L,\tau}] g\rVert_{L^2}^2 \leq CL^{-2}\lVert g \rVert^2_{H^k}.
  %   \end{equation*}
  %   Letting $g = f_\lambda$ and using regularity theory for elliptic differential equations \cite{Taylor2011}, it must hence hold that
  %   \begin{equation}\label{eq: supercell_convergence_qp_5}
  %       \lVert[A_\theta, h_{L,\tau}] f_\lambda \rVert_{L^2}^2 \leq C L^{-2}(\lvert\lambda\rvert + \frac{\epsilon}{2} + 1)^{2}.
  %   \end{equation}
  % Combining equations \eqref{eq: supercell_convergence_qp_4} and \eqref{eq: supercell_convergence_qp_5} yields
  % \begin{equation}\label{eq: supercell_convergence_qp_6}
  %     \int_\R\lVert (A_\theta - \lambda)h_{L,\tau}f_\lambda\rVert^2\d\tau \leq (1 + \delta)\frac{\epsilon^2}{4} + (1 + \delta^{-1})C L^{-2}(\lvert\lambda\rvert + \frac{\epsilon}{2} + 1)^{2}.
  % \end{equation}
  Hence, there must exist some $\tau \in \R$ such that
  \begin{equation}\label{eq: supercell_convergence_qp_7}
      \lVert (A_\theta - \lambda)h_{L,\tau}f_\lambda\rVert^2\leq ((1 + \delta)\epsilon^2 + (1 + \delta^{-1})C L^{-2}(\lvert\lambda\rvert + \epsilon + 1)^{2})\lVert h_{L,\tau}f_\lambda\rVert^2.
  \end{equation}
  Define $\Tilde{A}_{\theta'}$ by
  \begin{equation}
      \Tilde{A}_{\theta'} = \sum_k a_k(x, \theta \tau + \theta'(x - \tau))\frac{d^k}{dx^k}.
  \end{equation}
  We have $\Tilde{A}_{\theta'} \in H(A_{\theta'})$ and it holds that 
  \begin{equation}\label{eq: supercell_convergence_qp_8}
      \lVert (\Tilde{A}_\theta' - \lambda)h_{L,\tau}f_\lambda \rVert \leq \lVert (A_\theta - \lambda)h_{L,\tau}f_\lambda \rVert + \lVert (\Tilde{A}_{\theta'} - A_\theta)h_{L,\tau}f_\lambda \rVert.
  \end{equation}
  Observe that for $x \in [\tau - L, \tau + L]$ the following bound holds
  \begin{equation}\label{eq: supercell_convergence_qp_9}
     \lvert a_k(x, \theta'x + (\theta - \theta')\tau)-a_k(x,\theta x)\rvert \leq \lVert \nabla a_k \rVert_{\mathcal{C}^0} \lvert \theta - \theta'\rvert L.
  \end{equation}
  It follows that 
  \begin{equation}\label{eq: supercell_convergence_qp_10}
      \lVert (\Tilde{A}_{\theta'} - A_\theta)h_{L,\tau}f_\lambda \rVert_{L^2} \leq C \lvert \theta - \theta'\rvert L \lVert h_{L,\tau}f_\lambda \rVert_{H^k}.
  \end{equation}
  Since
  \begin{equation*}
      \lVert h_{L,\tau}f_\lambda \rVert_{H^k} \leq C \lVert f_\lambda\rVert_{H_k},
  \end{equation*}
  as $L$ increases, we find that 
  \begin{equation*}
     \begin{split}
          \lVert (\Tilde{A}_\theta' - \lambda)h_{L,\tau}f_\lambda \rVert &\leq C\bigg( \left((1 + \delta)\epsilon^2 + (1 + \delta^{-1}) L^{-2}(\lvert\lambda\rvert + \epsilon + 1)^{2}\right)^{\frac{1}{2}} \\ 
          &\hspace{5cm} + \lvert \theta - \theta'\rvert L(\lvert\lambda\rvert + \epsilon + 1)\bigg)\lVert f_\lambda \rVert.
     \end{split}
  \end{equation*}
  By choosing $L = \lvert \theta - \theta'\rvert^{-1/2}$ and letting $\epsilon \rightarrow 0$, the claim follows.
\end{proof}

While for $\theta \in \R\setminus\Q$ it holds that $\Sigma(A_\theta) = \sigma(A_\theta)$, the same is not true in general for $\theta\in\mathbb{Q}$. In order to provide proper justification for the supercell approach, a slightly different result is needed. In what follows, we will prove the following theorem:

\begin{theorem}\label{thrm: explicit_domain_qp}
    Let $\theta \in \R$ and $A_\theta$ a self-adjoint elliptic operator of the form \eqref{eq: operator_qp}. Let $\theta_l = \frac{p_l}{q_l}$ be the $l$\textsuperscript{th} continued fraction approximant of $\theta$. Then there exists a constant $C>0$ such that 
    \begin{equation*}
        \begin{split}
            d(\lambda, \sigma(A_\theta)) \leq  C(1 + \lvert \lambda\rvert)\lvert q_l^{-1}, \quad \lambda \in \sigma(A_{\theta_l})\\
            d(\lambda, \sigma(A_{\theta_l})) \leq C(1 + \lvert \lambda \rvert)\sum_{k = n}^\infty q_k^{-1}\prod_{l = n}^{k-1}(1+Cq_{l}^{-1}), \quad \lambda \in \sigma(A_\theta).
        \end{split}
    \end{equation*}
\end{theorem}

The idea is similar to that of Theorem \ref{thrm: implicit_domain_qp}, the key difference being that we explicitly construct the approximate eigenfunctions and hence control their localisation. The proof of Theorem \ref{thrm: explicit_domain_qp} relies on the following lemma:

\begin{lemma}\label{lem: explicit_domain_qp}
    Under the assumptions and in the notation of Theorem \ref{thrm: explicit_domain_qp}, there exists some $C>0$ such that for every $\lambda \in \sigma(A_{\theta_l})$ and every $m>l$ we have
    \begin{equation*}
        d(\lambda, \sigma(A_{\theta_m})) \leq  C(1 + \lvert \lambda\rvert)\lvert \theta - \theta_l\rvert^{-\frac{1}{2}}.
    \end{equation*}
\end{lemma}

\begin{proof}
    Let $\lambda \in \sigma(A_{\theta_l})$, from Floquet-Bloch theory we know that there exists some $k \in [0, \frac{2\pi}{q_l})$ and some periodic function $p(x)$ such that $v = e^{ikx}p(x)$ satisfies
    \begin{equation}\label{eq:Aeigen}
        A_{\theta_l}v(x) = \lambda v(x).
    \end{equation}
    We further define
    \begin{equation}\label{eq: transfer_function}
        f(x) = 
         \begin{cases}
          e^{-\frac{1}{x}},& x>0,\\
          0,& x\leq 0.
        \end{cases}
    \end{equation}
    It can be shown that $f$ is smooth. This function will serve as the primary building block for the following step- and bump functions. Consider now
    \begin{equation}\label{eq: step_function}
        g(x) = \frac{f(x)}{f(x) + f(1-x)}.
    \end{equation}
    As the denominator of $g$ never vanishes, $g$ is smooth and it holds that
    \begin{equation*}
        g(x)
        \begin{cases}
            1,& x\geq 1,\\
            0,& x \leq 0.
        \end{cases}
    \end{equation*}
    We now define
    \begin{equation*}
        h_l(x) = g\left(\frac{x-2q_l}{q_l}\right)g\left(\frac{2q_l-x}{q_l}\right).
    \end{equation*}
    It is not difficult to see that $h_l$ is smooth, $h_l\equiv1$ on $(-q_l, q_l)$, $h_l\equiv0$  outside of $[-2q_l, 2q_l]$ and that $\rvert h_l^{(n)}(x)\lvert \leq C_n q_l^{-n}$. Moreover, there exist $C_1,C_2>0$, independent of $l$ such that
    \begin{equation}\label{eq: sobolev_bound}
        C_1 \lVert p \rVert_{H^k((0,q_l))}^2 \leq \lVert h_l v \rVert_{H^k(\R)}^2 \leq C_2 \lVert p \rVert_{H^k((0,q_l))}^2.
    \end{equation}   
    We now have for every $\delta>0$
    \begin{equation}
        \int_\R \lvert(A_{\theta_l} - \lambda)h_l v\rvert^2 \d x = (1 + \delta)\int_\R \lvert h_l(A_{\theta_l} - \lambda)v\rvert^2 \d x + (1 + \delta^{-1})\int_\R \lvert[A_{\theta_l} - \lambda,h_l]v\rvert^2 \d x.
    \end{equation}
    As $v$ satisfies \eqref{eq:Aeigen}, the first term on the right-hand side of the previous equation vanishes. From \eqref{eq: supercell_convergence_qp_3}, we infer that
    \begin{equation}\label{eq: explicit_domain_qp_3}
        \lVert[A_{\theta_l}-\lambda, h_l] v\rVert_{L^2(0,q_l)}^2 \leq Cq_l^{-2}\lVert v \rVert^2_{H^k((0,q_l)},
    \end{equation}
    for some $C>0$ that does not depend on $l$. Due to elliptic regularity \cite{Taylor2011}, there exists some constant $C>0$ that can be chosen independently of $l$ such that 
    \begin{equation}\label{eq: explicit_domain_qp_4}
        \lVert p \rVert_{H^k(0,q_l)}^2\leq C (1 + \lvert \lambda \rvert)^2\lVert p \rVert_{L^2(0,q_l)}^2.
    \end{equation}
    Equation \eqref{eq: explicit_domain_qp_3}, \eqref{eq: sobolev_bound} and \eqref{eq: explicit_domain_qp_4} imply together that
    \begin{equation}
        \lVert (A_{\theta_l}-\lambda)h_lv\rVert^2 \leq C (1 + \lvert \lambda \rvert)^2q_l^{-2}\lVert h_lv \rVert^2,
    \end{equation}
    for  $C>0$, independent of $l$ by letting $\delta \rightarrow \infty$. Since for all $i$ we have
    \begin{equation}
        \lvert a_i(x,\theta x) - a_i(x, \theta_l)\rvert \leq \sup_{y} \lvert \nabla a_i(y) \rvert \lvert \theta_m - \theta_l\rvert  \lvert x \rvert,
    \end{equation}
    it follows that 
    \begin{equation}
        \lVert(A_{\theta_m} - A_{\theta_l})h_l v\lVert^2 \leq C(1 + \lvert \lambda \rvert)^{2}q_l^{-2}\lVert h_l v\rVert^2,
    \end{equation}
    where we used that $\lvert \theta_m - \theta_l\rvert \leq \frac{1}{q_lq_{l+1}}$. Altogether, this yields
    \begin{equation}
        \begin{split}
            \lVert (A_{\theta_m} - \lambda)h_l v \rVert ^2 &\leq (1 + \delta)\lVert(A_{\theta_m} - A_{\theta_l})h_l v\rVert^2 + (1 + \delta^{-1})\lVert (A_{\theta_l} - \lambda)h_l v\rVert ^2 \\
            &\leq C(1 + \lvert \lambda \rvert)^{2}q_l^{-2}\lVert h_l v \rVert^2.
        \end{split}
    \end{equation}
    This concludes the proof.
\end{proof}

\begin{proof}[Proof of Theorem \ref{thrm: explicit_domain_qp}]
    The first inequality of Theorem \ref{thrm: explicit_domain_qp} immediately follows from Lemma \ref{lem: explicit_domain_qp} by letting $m \to \infty$. For the converse direction, we need to keep track of how eigenvalues propagate upwards. Let $\lambda \in \sigma(A_{\theta_{l+2}})$, then by Lemma \ref{lem: explicit_domain_qp} we have
    \begin{equation}
        d(\lambda, \sigma(A_{\theta_{l+1}})) \leq C(1 + \lvert \lambda \rvert) q_{l+1}^{-1}.
    \end{equation}
    Hence, there exists some $\Tilde{\lambda} \in \sigma(A_{\theta_{l+1}})$ such that $d(\lambda, \Tilde{\lambda})\leq C(1 + \lvert \lambda \rvert)^k q_{l+1}^{-1} $. Therefore
    \begin{equation}
        \begin{split}
             d(\lambda, \sigma(A_{\theta_{l}})) &\leq  d(\Tilde{\lambda}, \sigma(A_{\theta_{l}})) + d(\lambda, \Tilde{\lambda})\\
             &\leq C(1 + \lvert \Tilde{\lambda} \rvert) q_{l}^{-1} + C(1 + \lvert \lambda \rvert) q_{l+1}^{-1}\\
             &\leq C(1 + \lvert \lambda \rvert + C(1 + \lvert \lambda \rvert) q_{l+1}^{-1}) q_{l}^{-1} + C(1 + \lvert \lambda \rvert) q_{l+1}^{-1}.
        \end{split}
    \end{equation}
    If we denote by $d_n$ the Hausdorff distance from $\lambda$ to $\sigma(A_{\theta_{n}})$,  with $d_0 = 0$, then this yields the recursion
    \begin{equation}
        d_{n - 1} \leq C(1 + \lvert \lambda \rvert + d_n) q_{n}^{-1} + d_{n},
    \end{equation}
    which admits the upper bound
    \begin{equation}
        d(n) \leq C(1 + \lvert \lambda \rvert)\sum_{k = n}^\infty q_k^{-1}\prod_{l = n}^{k-1}(1+Cq_{l}^{-1}), 
    \end{equation}
    where we use the convention $\prod_{l = n}^{n-1} = 1$. Note that this expression is indeed finite since $q_n$ grows at least as fast as the Fibonacci sequence. The claim now follows from Theorem~\ref{prop: spectrally_inclusive}.
\end{proof}

\subsection{Tiling-based operators}

A similar result holds for tiling-based operators, in this case the rate of convergence depends on the inflation rule. We treat the example of the golden mean Fibonacci tiling. To this end, we need the following Lemma.
\begin{lemma}\label{lem: fibonacci_finite_bits}
    Let $\mathcal{S}$ be a finite subsection of $\mathcal{F}_\infty$, such that all tiles remain intact. Then for any $n$ such that $\lvert \mathcal{S}\rvert \leq \lvert \mathcal{F}_{n-1} \rvert $, $\mathcal{S}$ can be found in the periodic arrangement of $\mathcal{F}_n$.
\end{lemma}
\begin{proof}
    We are going to proceed by induction. Since $\lvert\mathcal{S}\rvert < \infty$, there exists some $N \in \N$ such that $\mathcal{S}$ is contained in $\mathcal{F}_N$, and hence in particular in a periodic arrangement of $\mathcal{F}_N$. 

    Suppose the statement holds for $n+1$ and assume that $\lvert\mathcal{S}\rvert \leq \lvert \mathcal{F}_{n-1}\rvert $. We are going to discern two cases. If $\mathcal{S}$ is contained within $\mathcal{F}_{n+1}$, the statement is trivial since $\mathcal{F}_{n+1}$ is contained in the periodic arrangement of $\mathcal{F}_{n}$. On the other hand, if $\mathcal{S}$ is not a subset of $\mathcal{F}_{n+1}$, it must lie close to the boundary of two subsequent $\mathcal{F}_{n+1}$ cells. In particular, since $\lvert\mathcal{S}\rvert \leq \lvert \mathcal{F}_{n-1}\rvert $, $\lvert \mathcal{S}\rvert $ must be contained in two subsequent cells of $\mathcal{F}_{n-1}$, which in turn are contained in a periodic arrangement of $\mathcal{F}_n$.
\end{proof}
We can now state the supercell convergence Theorem for golden mean Fibonacci tilings.
\begin{theorem}\label{thrm: supercell_convergence_fibonacci}
    Let $A_n$ be a sequence of self-adjoint elliptic operators with coefficients based on the golden mean Fibonacci tiling as outlined in Subsection \ref{subsec: tilings}. Denote $A_\infty$ the limiting operator based on the even numbered tiles. There exists some $C>0$ such that the following hold
    \begin{equation*}
        \begin{split}
            d(\lambda, \sigma(A_\infty)) &< C(1 + \lvert \lambda \rvert)\mathcal{F}_{2n-1}, \quad \text{for all } \lambda \in \sigma(A_n),\\
            d(\lambda, \sigma(A_n)) &< C(1 + \lvert \lambda \rvert)\mathcal{F}_{2n-1}, \quad \text{for all } \lambda \in \sigma(A_\infty).
        \end{split}
    \end{equation*}
\end{theorem}
\begin{proof}
    The proof is similar to that of Theorem~\ref{thrm: implicit_domain_qp}, however instead of looking for areas where the coefficients of $A_n$ and $A_\infty$ are close, we search for areas where they are identical. By Lemma~\ref{lem: fibonacci_finite_bits}, these are exactly the finite sections of size $\mathcal{F}_{2n-1}$. Let $h$ be as in the proof of Theorem~\ref{thrm: implicit_domain_qp}. We let $L = \mathcal{F}_{2n-1}$ and $\tau = 0$. The estimate is then identical to that of Theorem~\ref{thrm: explicit_domain_qp}, with the difference that the second term on the right-hand side of \eqref{eq: supercell_convergence_qp_8} is zero.
\end{proof}

Tiling based structure make for excellent models for computational and experimental studies, as their main band gaps are quickly identified. The following recurrence relation was discovered in \cite{Kohmoto1983}.

\begin{lemma}\label{lem: trace_recursion}
    Let $T_n(\lambda)$ denote the transfer matrix associated to $A_n - \lambda$ and define $x_n(\lambda) = \operatorname{tr}(T_n(\lambda))$. The following recursion relation holds for $n\geq 2$:
    \begin{equation*}
        x_{n+1}(\lambda) = x_{n}(\lambda)x_{n-1}(\lambda) - x_{n-2}(\lambda).
    \end{equation*}
\end{lemma}

In \cite{Morini2018} it was observed that the main band gaps of operators based on Fibonacci tilings occur relatively early on in the sequence and persist in the limit of large unit cell size. This lead to the notion of \textit{super band gaps} $\mathcal{S}_N$, the set of eigenvalues that are in spectral gaps for all the operators with an index greater than $N$:
\begin{equation*}
    \mathcal{S}_N = \left\{\lambda \in \R\mid \lambda\notin\sigma(A_n) \text{ for all }
    n\geq N\right\}.
\end{equation*}
Using Lemma~\ref{lem: trace_recursion}, we can see that $\lambda$ being in a super band gap is equivalent to $\lvert x_n(\omega)\rvert >2$ for all $n$ greater than some $N$. The following criterion for deciding whether some $\lambda$ lies in a super band gap was proven in \cite[Theorem~3.1]{davies2023super}.

\begin{theorem}
    Let $\lambda \in \R$ and suppose $x_n(\lambda)$ satifies the recurrence relation of Lemma~\ref{lem: trace_recursion}. Suppose that there exists some $N\in \N$ such that
    \begin{equation}
        \lvert x_N(\lambda)\rvert >2, \lvert x_{N+1}(\lambda) \lvert \geq \lvert x_N(\lambda)\rvert \text{ and } \lvert x_{N+2}(\lambda) \lvert \geq \lvert x_{N+1}(\lambda)\rvert.
    \end{equation}
    Then $\lvert x_{n+1}\rvert > \lvert x_n \rvert$ for all $n \geq N$. Consequently, $\lvert x_n \rvert > 2$ for all $n \geq N$, meaning $\lambda$ lies in the super band gap $\mathcal{S}_N$.
\end{theorem}

If we assume our tiles to be piece-wise constant, as is assumed in the widely studied setting of laminates, the following refinement can be obtained.

\begin{corollary}
    Let $\omega\in\R$ and consider $x_n(\omega)$ satisfying the golden mean recursion. Then $\omega$ lies in the super band gap $\mathcal{S}_N$ if and only if there exists $n \leq N$ such that two consecutive values $\lvert x_n \rvert, \lvert x_{n+1}\rvert > 2+\epsilon$ for some $\epsilon>0$.
\end{corollary}
\begin{proof}
    Let $n$ be such that $\lvert x_{n+1} \rvert, \lvert x_{n+2}\rvert > 2+\epsilon$. Since $\lvert x_1(\omega)\rvert = 2\lvert \cos\left(\omega\right)\rvert \leq 2 $, we may assume without loss of generality that $\lvert x_n \rvert \leq 2$, or we reduce $n$ by one. We then have
    \begin{equation}
        \begin{split}
            \lvert x_{n+3} \rvert &\geq \lvert x_{n+2}\rvert \lvert x_{n+1} \rvert - \lvert x_n\rvert\\
            &\geq \lvert x_{n+2}\rvert \lvert x_{n+1} \rvert - \lvert x_{n+2}\rvert\\
            &\geq \lvert x_{n+2} \rvert \left(1+\epsilon\right)
        \end{split}
    \end{equation}
    and by induction it follows that $\lvert x_N \rvert > 2$ for $N>n$. The converse direction immediately follows from the definition of super band gaps.
\end{proof}

\subsection{Numerical examples} \label{sec:superspace_numerics}
For illustrative purposes, we begin with the most well-understood example of a Schr\"odinger operator with a quasiperiodic potential. We define
\begin{equation}
    V(x,y) = \sin( 2 \pi x) +  \sin(2 \pi y),
\end{equation}
and consider the eigenvalue problem
\begin{equation} \label{eq: schrodinger_problem}
    A_\theta u =\lambda u,
\end{equation}
for the operator
\begin{equation}
    A_\theta = - \Delta + V(x,\theta x),
\end{equation}
where $\theta =  (1+\sqrt{5})/2$.
\begin{figure}
     \centering
     \begin{subfigure}[b]{0.55\textwidth}
         \centering
         \includegraphics[height=2.7cm,clip,trim=8cm 0 7cm 0]{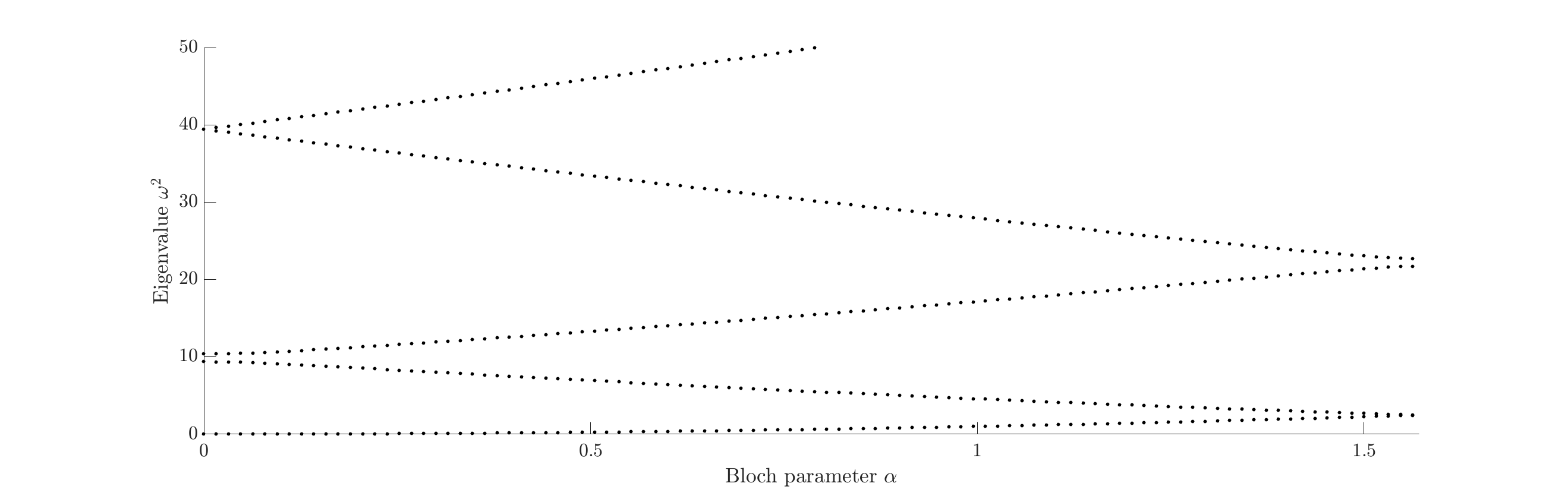}
         \caption{$\theta = \frac{3}{2}$ }
         \label{fig: qp_schrodinger_1}
     \end{subfigure}
     \begin{subfigure}[b]{0.245\textwidth}
         \centering
        \includegraphics[height=2.7cm]{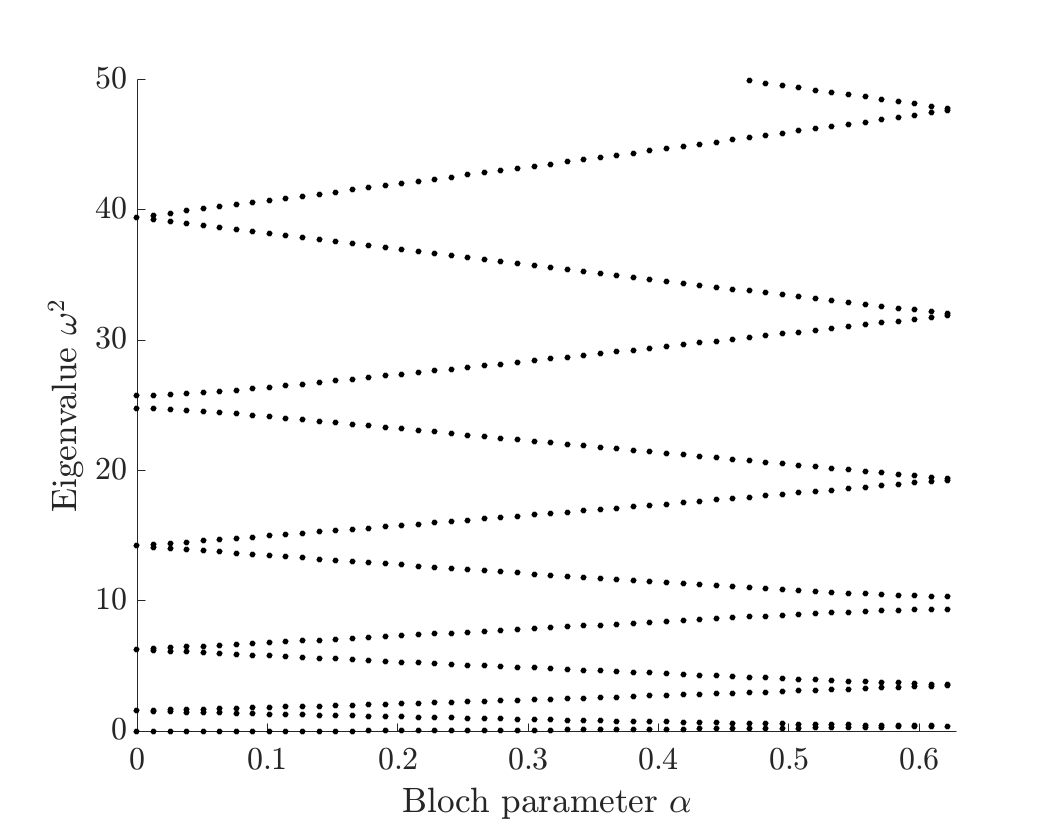}
         \caption{$\theta = \frac{8}{5}$}
         \label{fig: qp_schrodinger_2}
     \end{subfigure}
     \begin{subfigure}[b]{0.15\textwidth}
         \centering
        \includegraphics[height=2.7cm]{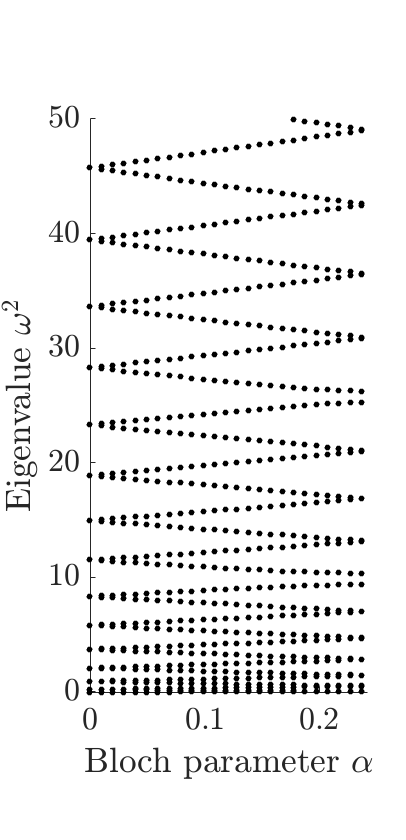}
         \caption{$\theta = \frac{21}{13}$}
         \label{fig: qp_schrodinger_3}
     \end{subfigure}
     \caption{The Floquet-Bloch spectra of a sequence of periodic approximants of the Schr\"odinger equation \eqref{eq: schrodinger_problem}.}
     \label{fig: qp_schrodinger}
\end{figure}
In Figure~\ref{fig: qp_schrodinger} we plot the band diagrams for a sequence of periodic approximants, where the periodic approximants are chosen according to the continued fraction expansion of $\theta$. For a more complex band diagram, we turn to the generalised eigenvalue problem
\begin{equation}\label{eq: generalised_problem}
    -\Delta f(x) = \lambda (V(x,\theta x) + 3) f(x).
\end{equation}
The resulting band diagrams are plotted in Figure~\ref{fig: qp_wave}. In this case, band gaps are visible, that are shared by all three of the supercell approximants considered. Assuming these continue to persist for subsequent supercell approximations, such that these are ``super band gaps'', Theorem~\ref{thrm: explicit_domain_qp} shows that there will be a corresponding band gap in the spectrum of the limiting quasiperiodic operator. In Figures~\ref{fig: qp_schrodinger} and~\ref{fig: qp_wave} it is apparent that, as the length of the periodic unit cell decreases, the slope of each spectral band decreases and the number of distinct bands in each fixed interval increases. This is precisely what must happen, as in the limit the Brillouin zone collapses to a set of measure zero.

\begin{figure}
     \centering
     \begin{subfigure}[b]{0.55\textwidth}
         \centering
         \includegraphics[height=2.7cm,clip,trim=8cm 0 7cm 0]{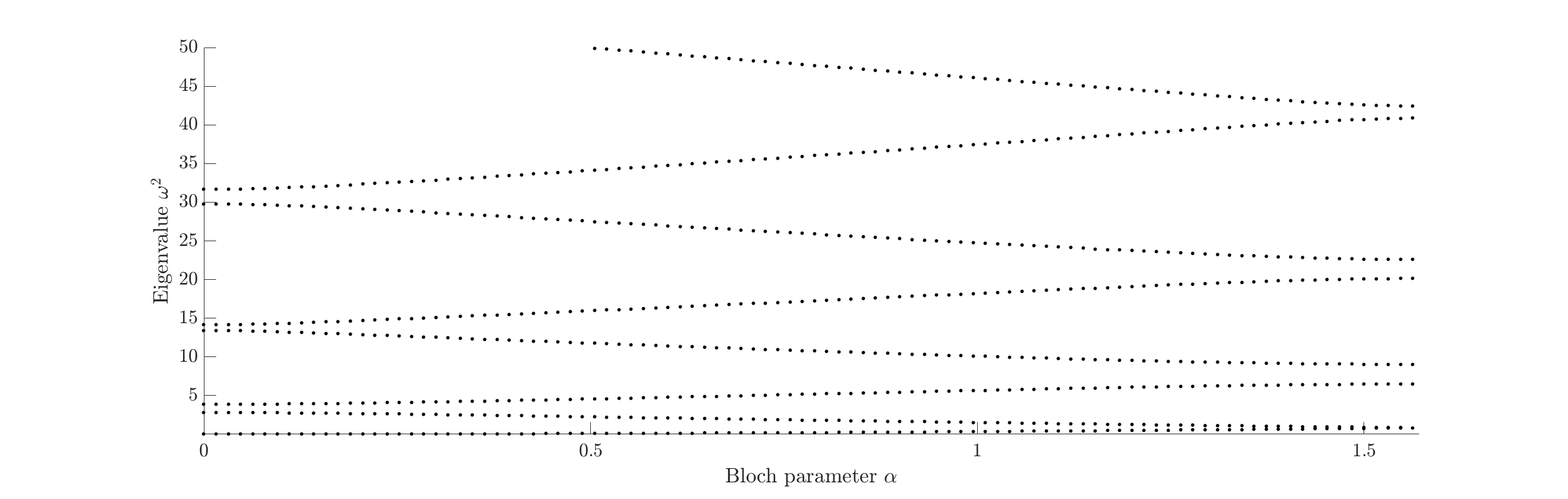}
         \caption{$\theta = \frac{3}{2}$}
         \label{fig: qp_wave_1}
     \end{subfigure}
     \begin{subfigure}[b]{0.245\textwidth}
         \centering
        \includegraphics[height=2.7cm]{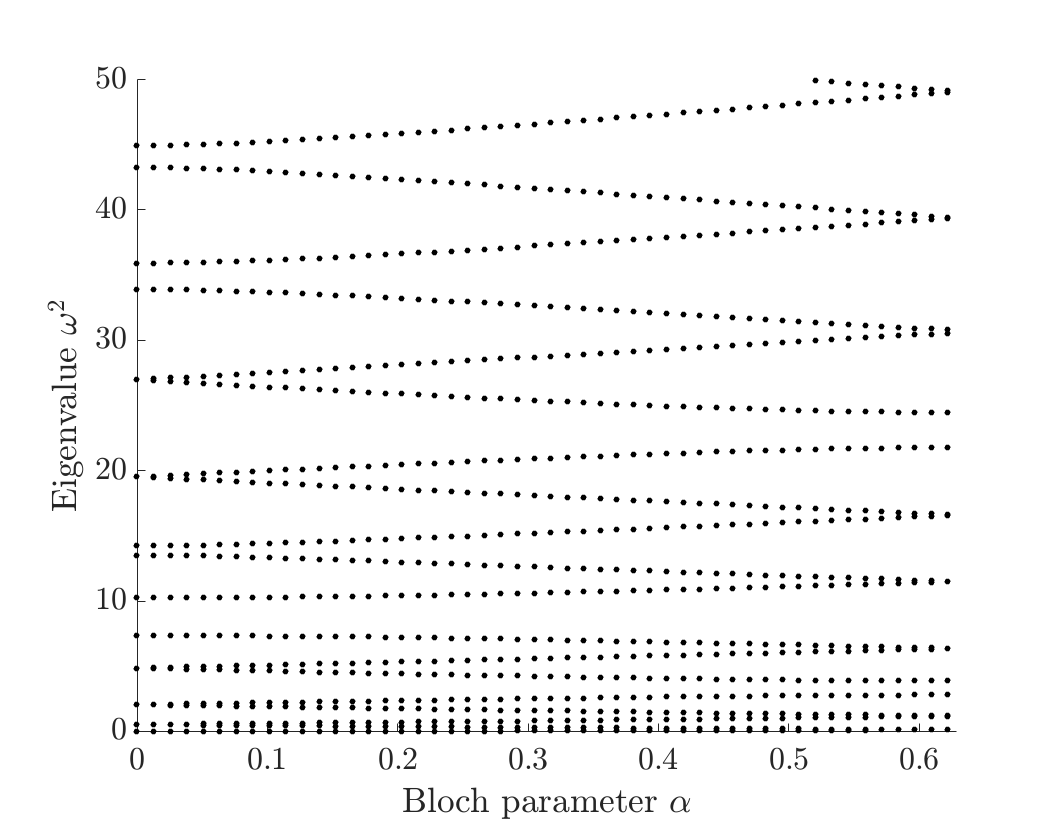}
         \caption{$\theta = \frac{8}{5}$}
         \label{fig: qp_wave_2}
     \end{subfigure}
     \begin{subfigure}[b]{0.15\textwidth}
         \centering
        \includegraphics[height=2.7cm]{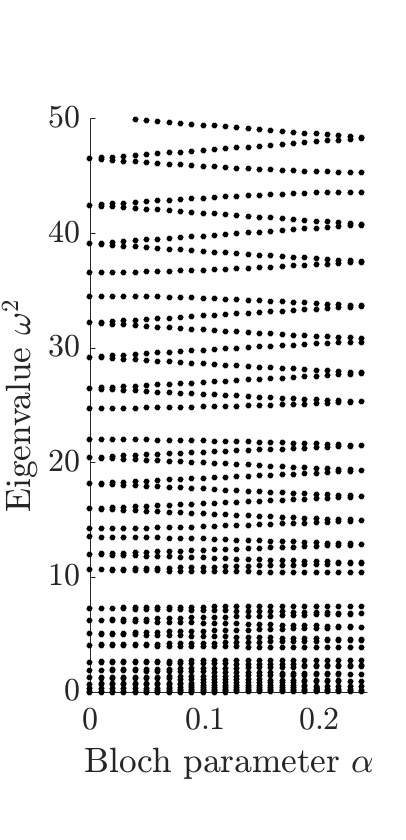}
         \caption{$\theta = \frac{21}{13}$}
         \label{fig: qp_wave_3}
     \end{subfigure}
     \caption{The Floquet-Bloch spectra of a sequence of periodic approximants of the generalised eigenvalue problem \eqref{eq: generalised_problem}.}
     \label{fig: qp_wave}
\end{figure}

\section{Superspace method}\label{sec: superspace}
A different approach to compute the spectrum is to exploit the fact that quasicrystals are periodic in a higher dimensional superspace: the "parent" space of the coefficients. Using this to perform numerical computations of the spectra was proposed by \cite{Rodriguez2008} and analogous ideas were also used to study forced problems by \cite{amenoagbadji2023wave}. 

\subsection{Theory}

In our example of the one-dimensional differential operator
\begin{equation}\label{eq: operator_qp_v2}
    A_\theta(x) = \sum_k a_k(x,\theta x) \frac{d^k}{dx^k},
\end{equation}
we lift into the two-dimensional space $\R^2$ and the new operator is
\begin{equation} \label{eq:liftedop}
    B_\theta = \sum_k a_k(x,y) D_\theta^k,
\end{equation}
where $D_\theta$ denotes the directional derivative along the vector $(1, \theta)^\top$ in $\R^2$. Since the coefficients of this operator \eqref{eq:liftedop} are periodic, Floquet-Bloch theory applies and one can reduce the spectral problem to studying the operators
\begin{equation}
    B_\theta(\alpha,\beta) = e^{-\i(\alpha x + \beta y)} \circ B_\theta \circ e^{\i(\alpha x + \beta y)}
\end{equation}
on $\mathbb{T}^2$.

It has been argued that, since a slice along an irrational angle fills the unit cell densely (see Figure \ref{fig: unit_cell_filled}), the eigenvalues of the above equation should not depend on $y$. Moreover, the two operators \eqref{eq: operator_qp_v2} and \eqref{eq:liftedop} should share the same spectrum. We will provide a more rigorous justification for these two statements. It has also been argued that to every $\lambda \in \sigma(B_\theta)$ there should exist Floquet eigenmodes
\begin{equation*}
    u_\lambda(x,y) = e^{\i(\alpha x + \beta y)}p(x,y),
\end{equation*}
where $p$ is periodic. However, to the best of the authors' knowledge, the existence of such functions usually requires ellipticity of the associated differential equations, which is not the case here. A priori, we can therefore not assume the existence of such solutions. Only for Schrodinger operators with analytic potential this has been rigorously proven, see \cite{Dinaburg1976}.

The next Theorem justifies the superspace method for spectral problems as it shows that the lifted operator has the same spectrum as the original quasiperiodic operator.
\begin{theorem}\label{thrm: superspace_equivalence}
    Let $\theta \in \R\setminus\Q$ and $A_\theta$ be a one-dimensional differential operator as in \eqref{eq: operator_qp_v2}. If $B_\theta$ is the lifted operator defined in \eqref{eq:liftedop}, then $\sigma(A_\theta) = \sigma(B_\theta)$.
\end{theorem}
\begin{proof}    
    From Theorem \ref{thrm: spectrum_equivalence}, we have $\sigma_B(A_\theta) = \sigma(A_\theta)$. Let $\lambda \in \sigma(A_\theta)$ and $ \epsilon>0$. By Lemma~\ref{lem: approximate_eigenvectors} and Definition~\ref{def: besicovitch_space}, there exists an almost eigenvector $f \in B^2$ with generalised Fourier series
    \begin{equation}
        f(x) = \sum_{\gamma \in \R} f_\gamma e^{\i\gamma x},
    \end{equation}
    where only countably many $f_\gamma$ are nonzero. By considering how $A_\theta$ acts on the basis vectors $\left\{\exp(\i\gamma x)\right\}_{\gamma \in \R}$, we observe that $A_\theta$ preserves subspaces
    \begin{equation*}
        E_k = \left\{g(x) \in B^2(\R) \mid g(x) = e^{2\pi k x}\sum_{m,n}g_{mn}e^{2\pi \i(mx + \theta n x)} \text{ where }g_{mn} \in l^2(\Z^2)\right\}.
    \end{equation*}
     Hence there must exist some $k \in \R$ such that there exists an $\epsilon$-almost eigenvector
    \begin{equation}
        \Tilde{f}(x) = e^{\i kx} \sum_{m,n} f_{mn}e^{2\pi \i(m+\theta n)x}.
    \end{equation}
    We define 
    \begin{equation}
        \Tilde{F}(x,y) = e^{\i kx} \sum_{m,n} f_{mn}e^{2\pi \i(mx+\theta ny)}.
    \end{equation}
    Since $B_\theta$ acts on the Fourier coefficients of $\Tilde{F}$ like $A_\theta$ acts on the (generalised) Fourier coefficients of $\Tilde{f}$, $\Tilde{F}$ is an $\epsilon$ eigenvector for $B_\theta$. Hence $\sigma(A_\theta) \subset \sigma(B_\theta) $. The converse inclusion follows from a similar argument utilising the Fourier series of (approximate) eigenvectors of $B_\theta$.
\end{proof}
In the proof of Theorem \ref{thrm: superspace_equivalence}, it also becomes apparent why the spectrum of $B_\theta(\alpha)$ must be independent of the quasi-momentum $\alpha$. Since $\theta$ is irrational, any $\alpha$ can be approximated to arbitrary precision just by the Fourier modes of the approximate eigenvector.

% While the superspace method allows us to reduce the spectral problem of \eqref{eq: operator_qp} to a periodic one, the usual numeric techniques do not necessarily apply due to the loss of ellipticity. In particular, plane wave expansion methods seem to be rather unreliable, as we will see in the next subsection. This should not come as a surprise, for we cannot expect any reasonable decay estimate on the Fourier coefficients of $F$, as $m,n$ grow independently.
\subsection{Numerical Examples}

Theorem~\ref{thrm: superspace_equivalence} shows that the lifted operator has the same spectrum as the original quasiperiodic operator. To use this fact to compute spectra, all that remains is to choose a method to discretise the lifted operator. One way to achieve this is using the method of finite differences, as we did in Section~\ref{sec:superspace_numerics} for the supercell approach. However, some care is needed, as we need to take directional derivatives along a vector with incommensurate entries. If we were to naively use a square mesh, we would need to approximate these vectors, for example by the means of finite continued fractions. This would lead to complicated finite difference matrices. Instead, we opt to employ a rectangular mesh, where the ratio between the $x$ and $y$ step size is given by $\theta^{-1}$. If we index our grid points as $x_i = ih$ and $y_j = j\theta h$, then the second-order finite difference directional derivative of a function $f$ at point $(x_i,y_j)$ is given by
\begin{equation*}
    D_\theta^2 f(x_i,y_j) \approx \frac{f(x_{i+1},y_{j+1}) - 2f(x_i,y_j) + f(x_{i-1},y_{j-1})}{h^2}.
\end{equation*}

\begin{figure}
    \centering
    \begin{subfigure}[b]{0.49\textwidth}
        \includegraphics[width=\linewidth]{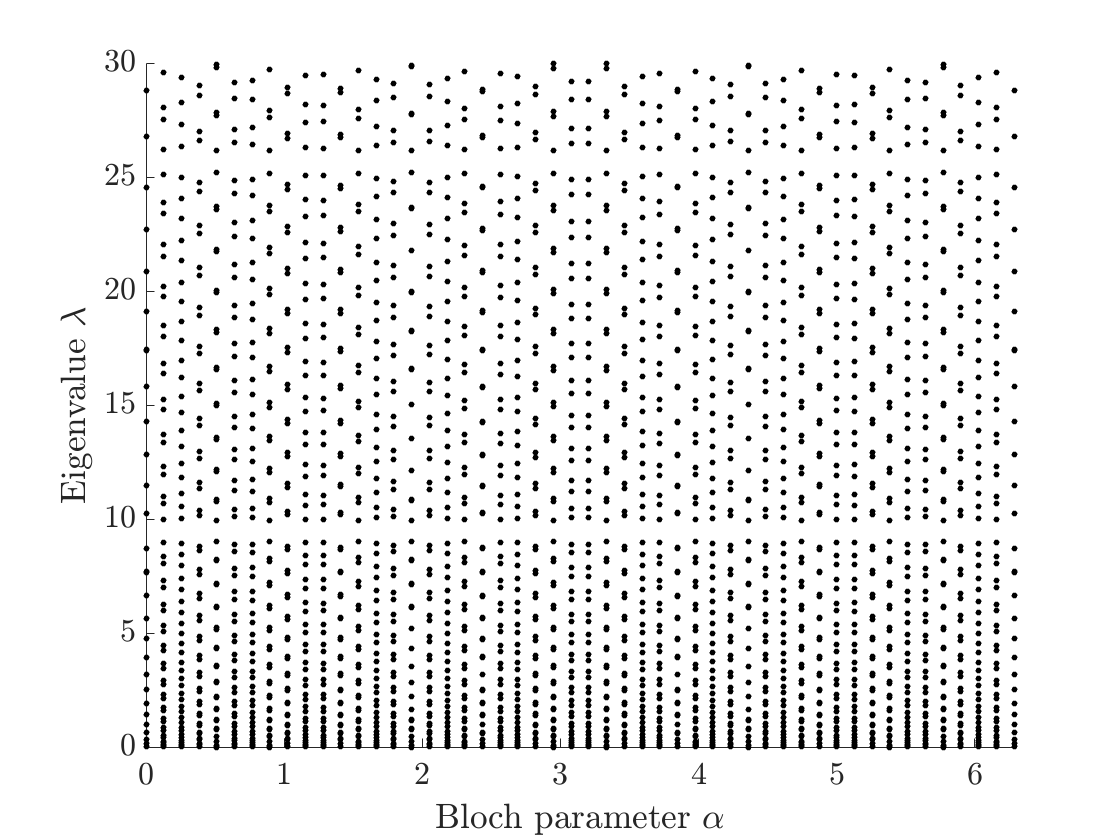}
        \caption{Schrodinger eigenvalue problem \eqref{eq: schrodinger_problem}.}
        \label{fig: superspace_band_diagram}
    \end{subfigure}
    \begin{subfigure}[b]{0.49\textwidth}
        \includegraphics[width=\linewidth]{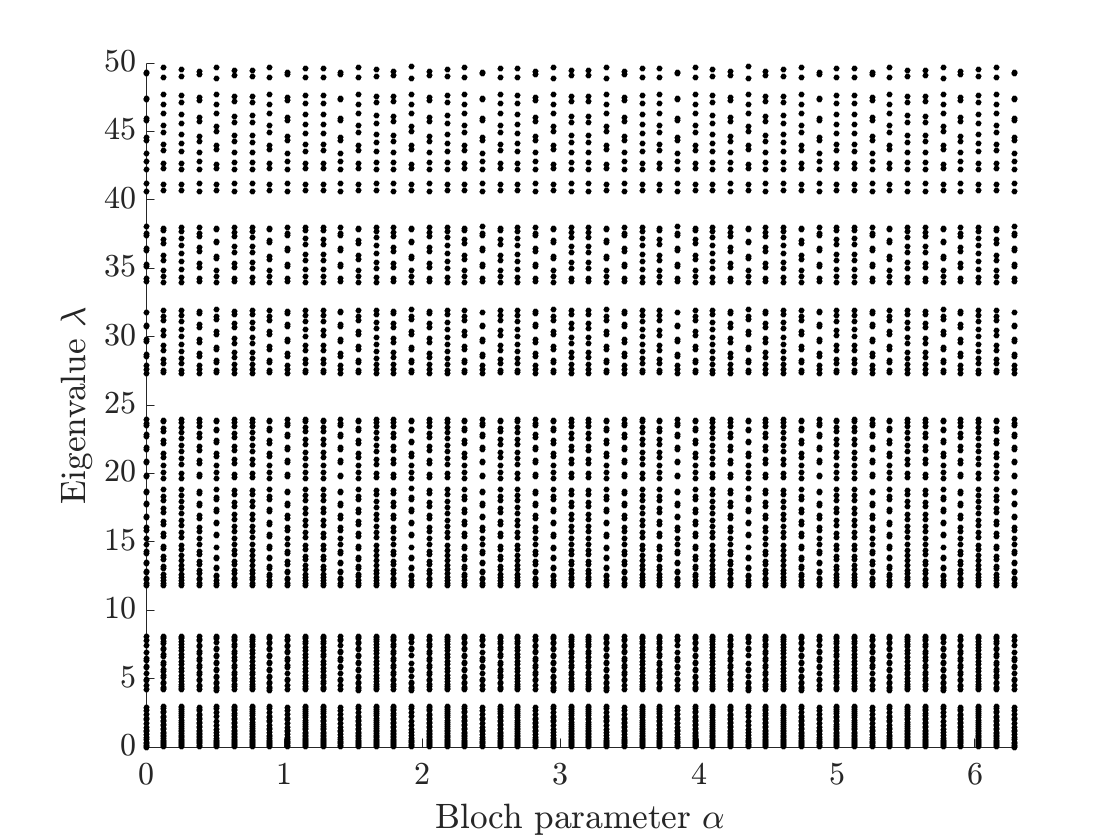}
        \caption{Generalised eigenvalue problem \eqref{eq: generalised_problem}.}
        \label{fig: superspace_band_diagram_wave}
    \end{subfigure}
    \caption{Band diagrams computed using the superspace approach with the finite difference method using a rectangular mesh of characteristic size $h = 0.02$.}
\end{figure}

In Figure \ref{fig: superspace_band_diagram} we plot the band diagram of the Schrodinger eigenvalue problem \eqref{eq: schrodinger_problem} using a mesh size of $h = 0.02$. The bands are shown along the $\alpha$ axis only, for convenience of visualisation. As the potential $V$ is analytic, we expect quasiperiodic eigenvectors, two of which are plotted in Figure \ref{fig: superspace_fd_eigenmodes}. The band diagram of the generalised eigenvalue problem \eqref{eq: generalised_problem} can be seen in Figure \ref{fig: superspace_band_diagram_wave}. In both cases, we observe good agreement with the spectral gaps computed using the supercell approach, shown in Figures~\ref{fig: qp_schrodinger} and~\ref{fig: qp_wave}.

\begin{figure}
     \centering
     \begin{subfigure}[b]{0.49\textwidth}
         \centering
         \includegraphics[width=\textwidth]{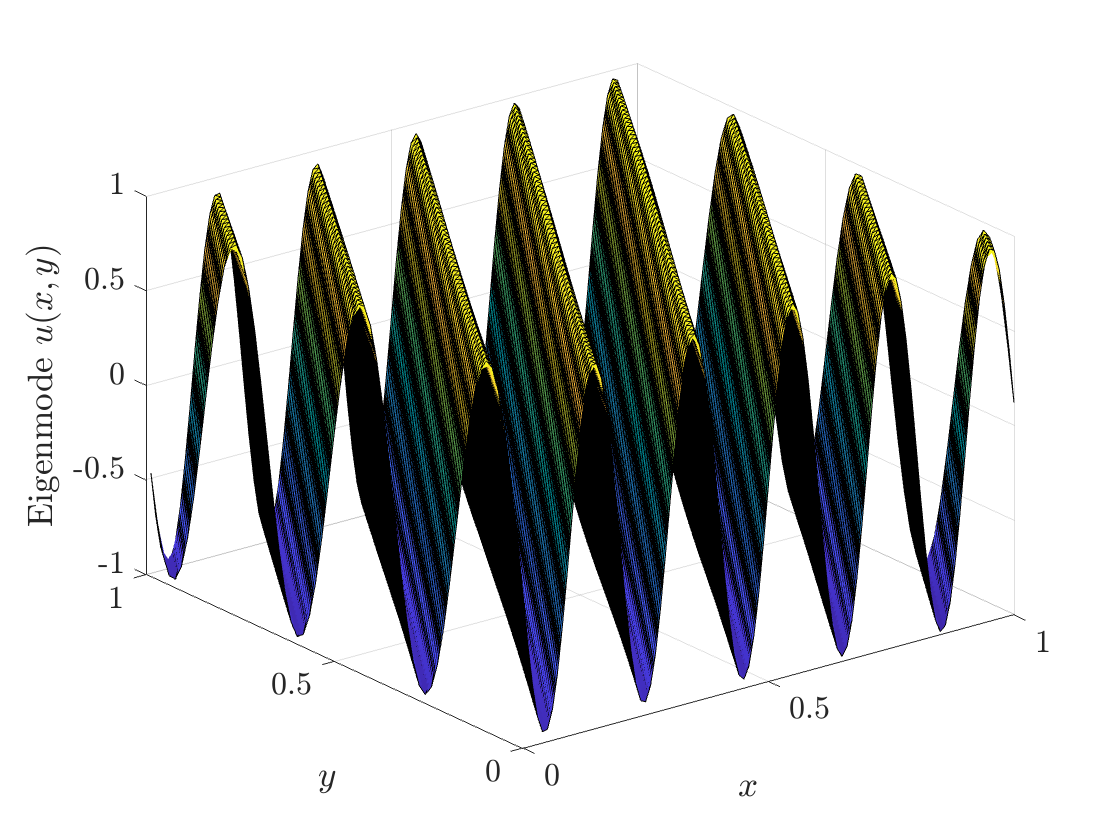}
         \caption{Eigenmode associated to $\lambda = 0.4770$.}
         \label{fig: superspace_fd_eigenmode_1}
     \end{subfigure}
     \hfill
     \begin{subfigure}[b]{0.49\textwidth}
         \centering
        \includegraphics[width=\textwidth]{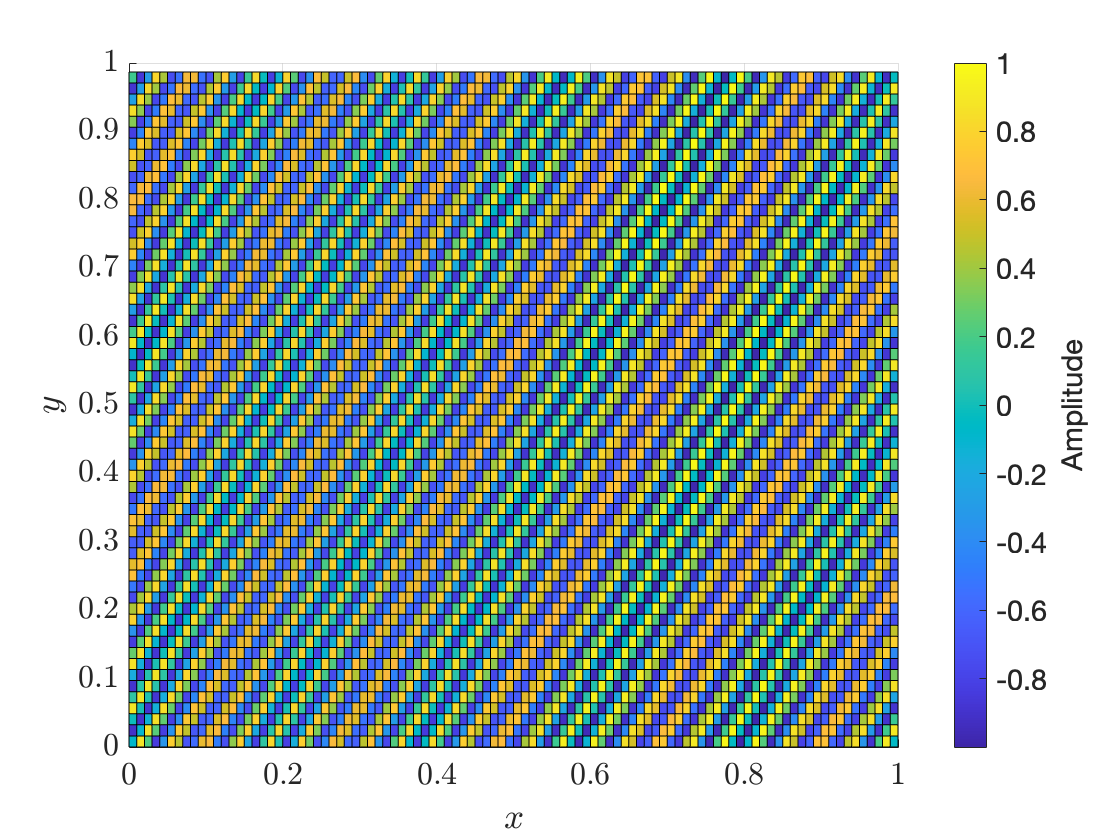}
         \caption{Eigenmode associated to $\lambda = 7.2844$.}
         \label{fig: superspace_fd_eigenmode_2}
     \end{subfigure}
        \caption{Two eigenmodes of the Schrodinger eigenvalue problem \eqref{eq: schrodinger_problem} computed using the superspace approach with a finite difference method and a rectangular mesh of size $h = 0.02$.}
        \label{fig: superspace_fd_eigenmodes}
\end{figure}

% \begin{figure}
%     \centering
%     \includegraphics[width=0.5\linewidth]{figures/superspace_fd_generalised_ev.png}
%     \caption{Band diagram of the generalised eigenvalue problem \eqref{eq: generalised_problem} computed using the superspace approach with a finite difference method and a rectangular mesh of size $h = 0.02$.}
%     % \label{fig: superspace_band_diagram_wave}
% \end{figure}

\begin{figure}
     \centering
     \begin{subfigure}[b]{0.49\textwidth}
         \centering
         \includegraphics[width=\textwidth]{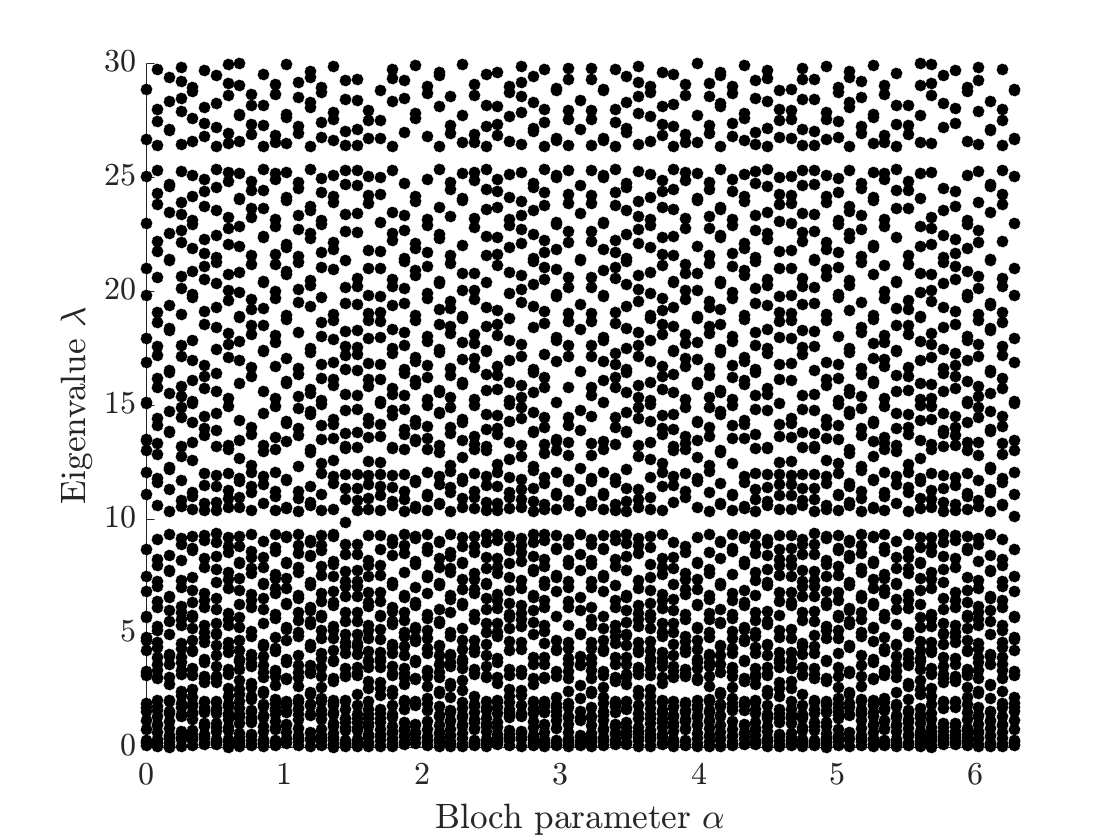}
         \caption{Dispersion relation of \eqref{eq: schrodinger_problem} calculated using the superspace method with $N=50$ plane waves.}
         \label{fig: superspace_schrodinger}
     \end{subfigure}
     \hfill
     \begin{subfigure}[b]{0.49\textwidth}
         \centering
        \includegraphics[width=\textwidth]{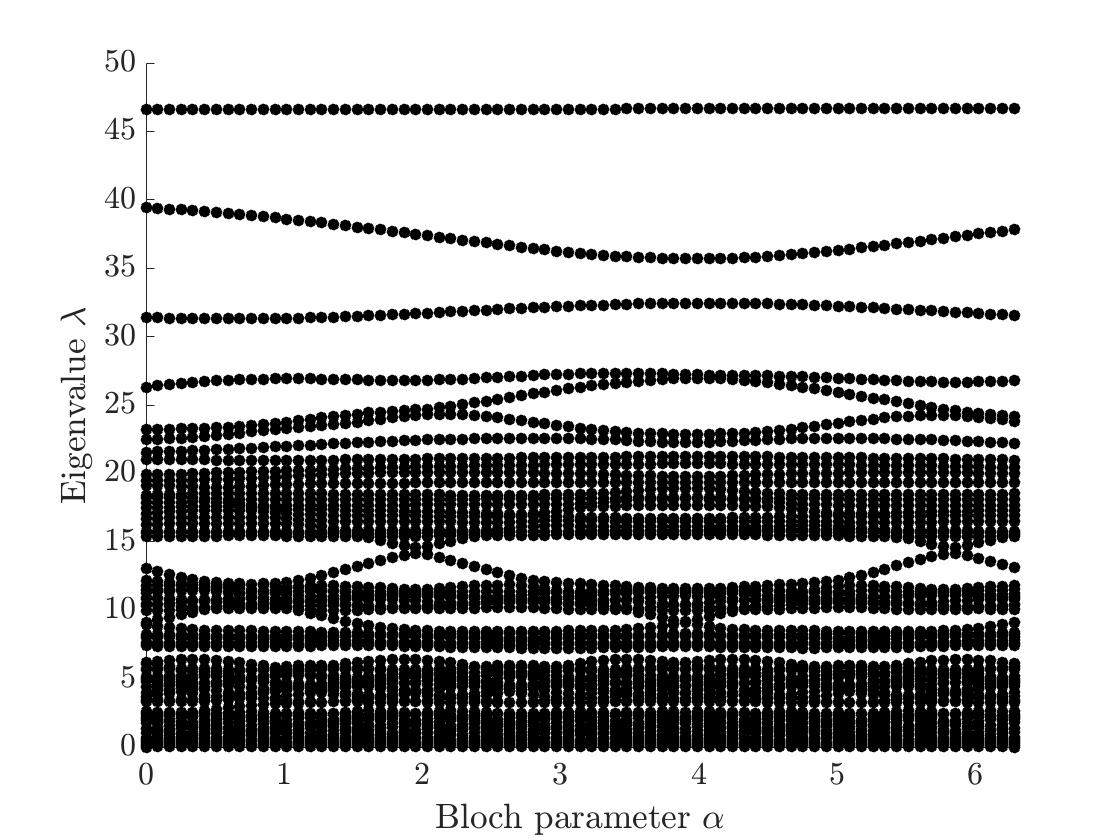}
         \caption{Dispersion relation of \eqref{eq: generalised_problem} calculated using the superspace method with $N=40$ plane waves.}
         \label{fig: superspace_generalised}
     \end{subfigure}
        \caption{Spectra computed using the superspace method with the plane wave expansion method used to discretise in lifted space.}
        \label{fig: superspace_plane_wave}
\end{figure}

Another way to implement the superspace approach is by implementing a plane wave expansion method to discretise the lifted operator, as suggested by \cite{Rodriguez2008}. In Figure~\ref{fig: superspace_plane_wave} we show the spectra of the spectral problem \eqref{eq: eigenvalue_problem} and the generalised spectral problem \eqref{eq: generalised_problem} computed using this approach. In this case, many of the main spectral gaps are obscured by the appearance of spurious eigenvalues. Similar spectral pollution was observed in \cite{Rodriguez2008}. While \cite{Rodriguez2008} proposed some formal arguments for identifying (and, hence, removing) spurious modes, this represents a significant weakness of the plane wave expansion method.

The failure of the plane wave expansion method to reliably predict the spectrum can most likely be attributed to the lack of ellipticity in the superspace, which means we cannot a priori expect any sort of decay of the Fourier coefficients. This can also be understood as the lack of Bragg resonances in this aperiodic material.  For periodic systems, given any fixed frequency $\omega$, there are only a finite number of plane waves with frequency $\omega' \leq \omega$ satisfying given Bloch boundary conditions. Since a Bloch eigenfunction $u(x)$ with eigenvalue $\omega^2$ must predominantly consist of plane waves with frequency less than $\omega$, truncating a plane wave expansion yields sensible results. On the other hand, in quasi-periodic systems, the lack of a countable series of spectral bands means that any eigenmode can be composed of an infinite number of plane waves.

\section{Localised interface modes} \label{sec:interface}
To highlight the importance of convergence statements like Theorem \ref{thrm: explicit_domain_qp}, we will show how it can be applied to the study of localised interface modes. These modes occur when an interface is introduced to a quasicrystal, for example by introducing an artificial axis of reflectional symmetry \cite{davies2022symmetry, MartSabat2021, Apigo2019, Liu2021}, by joining two different materials together \cite{verbin2013observation} or, simply, at the edges of finite-sized materials \cite{xia2020topological, kraus2012topological}. They are localised at this interface, in the sense that their amplitude decays exponentially as a function of distance from the interface. As a result, they allow for very strong focusing of wave energy and are the starting point of many different wave guiding and manipulating devices. Although localised interface modes have been studied extensively in periodic media with defects, analogous modes in quasicrystalline settings have also been studied previously, see for instance \cite{davies2022symmetry} for Fibonacci tiling based structures or \cite{Apigo2019,MartSabat2021,verbin2013observation,Liu2021,MartSabat2022} for more general quasicrystals. The complex spectra of quasiperiodic structures allow for many more bandgaps and hence more opportunity for defect-induced localisation.

In this section, we consider quasicrystalline materials in which interfaces have been formed by introducing an artificial axis of reflectional symmetry. That is, we consider operators of the form
\begin{equation}\label{eq: operator_qp_reflected}
    A_\theta(x)u := \sum_{k = 0}^2 a_k(\lvert x \rvert,\theta \lvert x \rvert) \frac{d^ku}{dx^k} = \omega^2 \rho u
\end{equation}
where $a_k,\rho \in \mathcal{C}^\infty_b(\mathbb{T}^2)$. The main result of this section is the following:
\begin{theorem}
     Let $\theta \in \R$ and $A_\theta$ a self-adjoint elliptic operator of the form \eqref{eq: operator_qp_reflected}. Let $\theta_l = \frac{p_l}{q_l}$ be the $l$\textsuperscript{th} continued fraction approximant of $\theta$. If there exists some $l_0$ such that $\lambda$ is an interface eigenvalue for $A_{\theta_{l_0}}$ that is sufficiently well-separated from the rest of the spectrum, i.e.
     \begin{equation}\label{eq: isolation_condition}
         d(\lambda, \sigma(A_{\theta_{l_0}})> 2C(1 + \lvert \lambda \rvert)\sum_{k = n}^\infty q_k^{-1}\prod_{l = n}^{k-1}(1+Cq_{l}^{-1}), \quad \text{for all } \lambda \in \sigma(A_\theta),
     \end{equation}
     then $A_\theta$ must also have a localised interface eigenmode.
\end{theorem}
\begin{proof}
    Existence of the limiting eigenvalue follows from Theorem \ref{thrm: explicit_domain_qp}, since the arguments apply to this settings as well. To see that this eigenvalue is isolated (and hence localised), we note that \eqref{eq: isolation_condition} ensures that $\lambda$ cannot merge with the spectral bands as $l\rightarrow \infty$. By utilizing the Liuoville-transform to bring equation \eqref{eq: operator_qp_reflected} in divergence form, we can see that the results of \cite{thiang2023bulk, Coutant2024, alexopoulos2023topologically} apply and hence it follows that any bandgap in a periodic structure can only support finitely many interface eigenvalues. Hence $\lambda$ must always remain isolated.
\end{proof}

\begin{figure}
     \centering
     \begin{subfigure}[b]{0.45\textwidth}
         \centering
         \includegraphics[width=\textwidth]{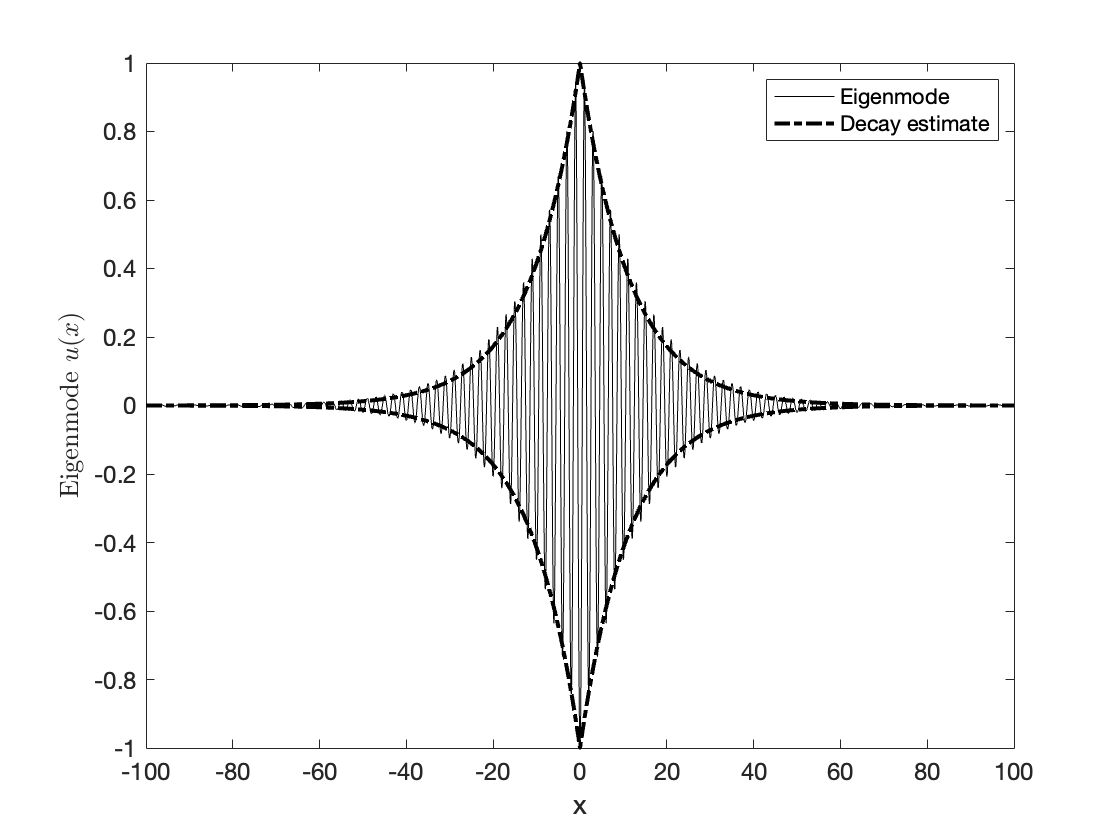}
         \caption{$\omega_1^2 = 9.9358$}
         \label{fig: interface_qp_1}
     \end{subfigure}
     \hspace{0.2cm}
     \begin{subfigure}[b]{0.45\textwidth}
         \centering
         \includegraphics[width=\textwidth]{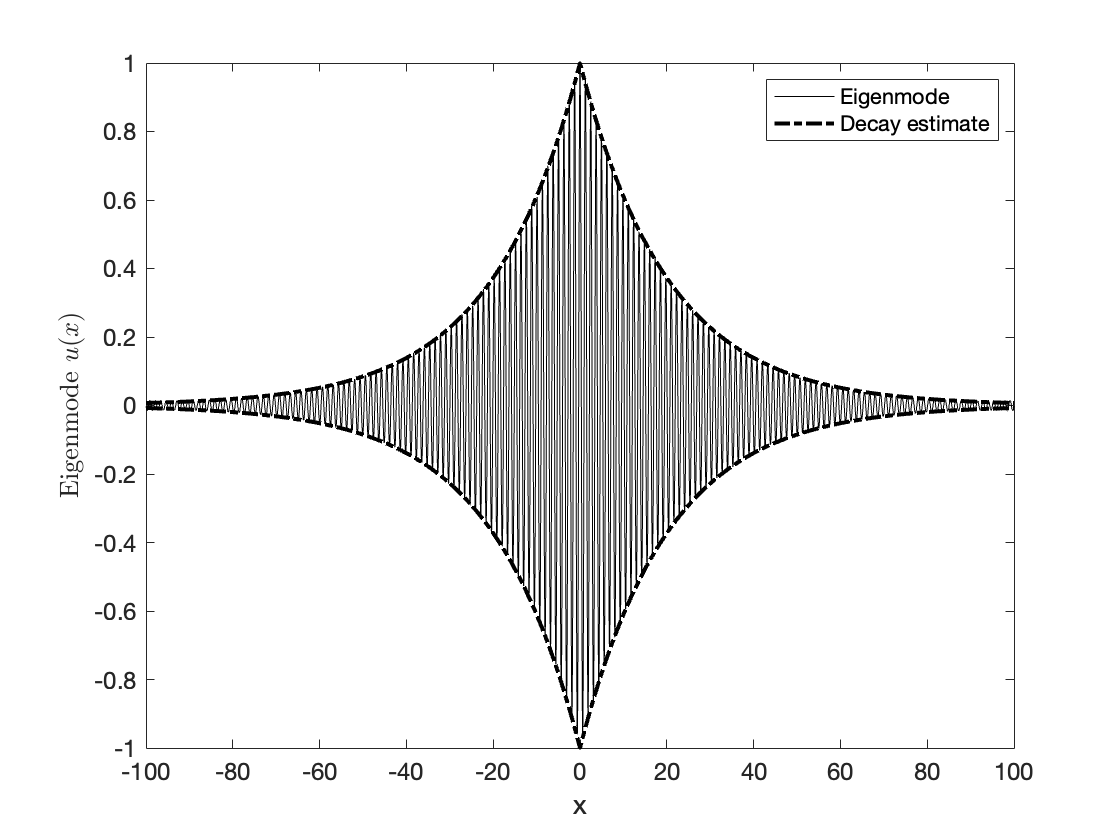}
         \caption{$\omega_2^2 = 25.8574$}
         \label{fig: interface_qp_2}
     \end{subfigure}
        \caption{The localised interface modes for \eqref{eq: reflected_schrodignger} decay exponentially away from the reflection-based interface. Alongside the modes the decay rate derived from the periodic approximation is shown as a dotted line.}
        \label{fig: interface_qp}
\end{figure}

As an example, consider a quasiperiodic Schrodinger operator with a reflection-induced interface
\begin{equation}\label{eq: reflected_schrodignger}
    A_\theta = -\frac{d^2}{dx^2} + \sin(2\pi\lvert x \rvert) + \sin(2\pi\theta\lvert x\rvert).
\end{equation}
For $k>l$ and $\theta = \frac{1 + \sqrt{5}}{2}$, the unit cell with index $l$ (up to errors of size $\mathcal{O}(q_{l+1}^{-1})$) is contained in the unit cell with index $k$ approximately $\theta^{k-l}$ times. Assuming that the mode decays uniformly across the unit cell, this suggests to consider an approximate decay rate given by
\begin{equation}\label{eq: improved_decay_estimate}
\exp\left(\min_{k}\left(\frac{\log(|\lambda_{k,min}|)}{q_k}\right)\lvert x\rvert \right),
\end{equation}
where $\lambda_{k,min}$ is the smaller of the two eigenvalues associated to the transfer matrix of the $k$\textsuperscript{th} continued fraction approximant. Note that, since we are in a band gap of these periodic approximants, we expect $|\lambda_{k,min}|<1$ so $\log(|\lambda_{k,min}|)<0$. A similar decay estimate for interface structures composed of Fibonacci tilings was derived in \cite{davies2022symmetry}. In Figure \ref{fig: interface_qp}, it can be seen that this expression indeed provides a very accurate approximation of the decay.

Similarly, we can also consider interface modes for the generalised eigenvalue problem with reflected quasiperiodic coefficients
\begin{equation}\label{eq: reflected_generalised}
    -\frac{d^2}{dx^2} u = \lambda (3+ \sin(2\pi\lvert x \rvert) + \sin(2\pi\theta\lvert x\rvert))u.
\end{equation}
Two resulting localised modes are plotted together with decay estimates in Figure~\ref{fig: interface_qp_generalised}.

\begin{figure}
     \centering
     \begin{subfigure}[b]{0.45\textwidth}
         \centering
         \includegraphics[width=\textwidth]{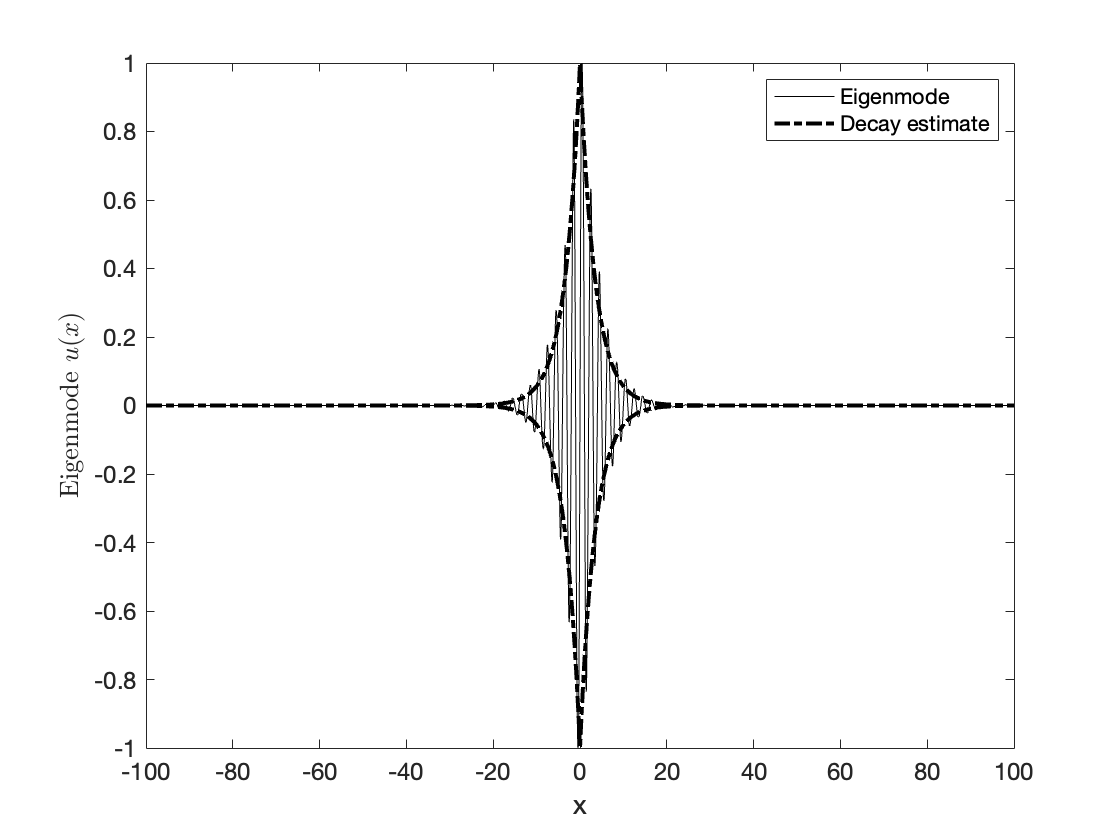}
         \caption{$\omega_1^2 = 3.2403$}
       
     \end{subfigure}
     \hspace{0.2cm}
     \begin{subfigure}[b]{0.45\textwidth}
         \centering
         \includegraphics[width=\textwidth]{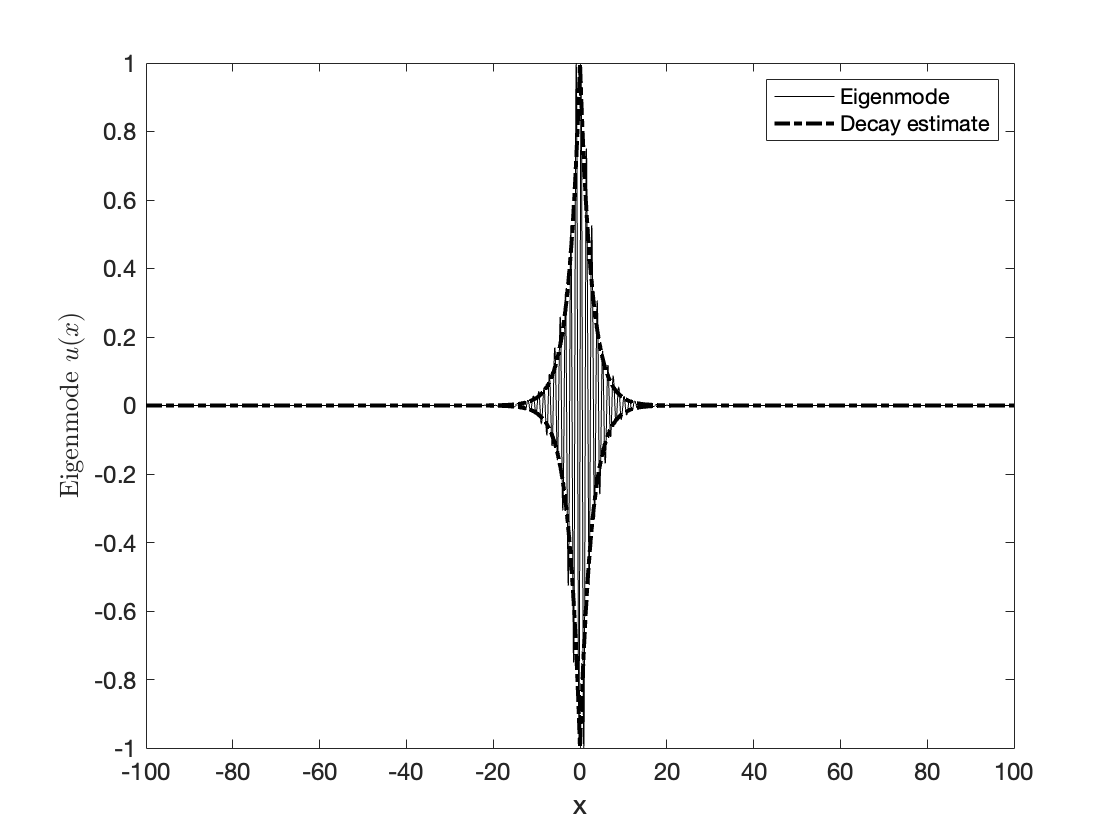}
         \caption{$\omega_2^2 = 8.1634$}
  
     \end{subfigure}
        \caption{The localised interface modes for the generalised eigenvalue problem \eqref{eq: reflected_generalised}. The decay rate derived from the periodic approximation is shown as a dotted line.}
        \label{fig: interface_qp_generalised}
\end{figure}

\section{Conclusion}

In this work, we have presented convergence theories for the supercell and superspace methods for approximating the spectra of quasiperiodic operators. We have shown that the series of Floquet-Bloch spectra predicted by supercell approximations converges to the limiting spectrum in a Hausdorff sense when the unit cell is arbitrarily large. This result builds on \cite{Damanik2016} by giving explicit constructions of the supercells needed to yield predicted convergence rates. This shows that when super band gaps appear in the supercell approximations \cite{Morini2018, davies2023super}, these must correspond to spectral gaps in the spectrum of the limiting operator. For the superspace method, we proved that the lifted operator has the same spectrum and showed how the spectral pollution observed in \cite{Rodriguez2008} can be mitigated by choosing the numerical discretisation method appropriately. Finally, we extended this theory to describe localised eigenmodes in systems that have been perturbed to have interfaces. We proved that such modes exist and used periodic approximations to gives estimates for their eigenfrequencies and decay rates. 

Our work helps to pave the way for the realisation of new quasiperiodic metamaterial devices. The theory discussed here provides rigorous foundations for using convenient Floquet-Bloch methodologies to approximate spectra. This will allow for more systematic investigation of quasiperiodic metamaterial design problems, using approaches that build on techniques that have been well established for periodic materials in metamaterial science. Our results are particularly useful for predicting the main spectral gaps and the eigenfrequencies of any localised modes, both of which are fundamental to wave guiding and manipulation applications.

\vskip6pt

\enlargethispage{20pt}

\section*{Acknowledgements}{The authors would like to thank Sebastien Guenneau, Richard Craster and Habib Ammari for their helpful suggestions. The authors declare that they have no competing interests and that they have not used AI-assisted technologies in creating this article.}

\printbibliography

\end{document}